%% file: paperPiotr22.tex
\documentclass[11pt]{article}
\usepackage{amssymb,amsmath,amsthm}
\usepackage{bbm}
\usepackage{graphicx}
\usepackage{color}
\usepackage[T1]{fontenc}
\usepackage[a4paper,margin=2.5cm]{geometry}  

\newtheorem{thm}{Theorem}
\newtheorem{lem}[thm]{Lemma}
\newtheorem{prop}[thm]{Proposition}

\def\C{\mathbb{C}}
\def\R{\mathbb{R}}
\def\P{\mathbb{P}}
\def\Q{\mathbb{Q}}
\def\E{\mathbb{E}}
\def\N{\mathbb{N}}
\def\diffd{\mathrm{d}}
\def\cF{\mathcal{F}}
\newcommand{\indic}[1]{\mathbbm{1}_{\{#1\}}}

\def\TW{{\omega}} 
\def\pr{{u}}		
\def\tpr{{\tilde u}}
\def\Nhits{{K}}
\def\auxsol{v}

\def\Nabs{\mathcal{N}_{\text{abs}}}
\def\Nabs{\mathcal{N}_{\text{live}}}
\def\Nall{{\mathcal{N}_{\text{all}}}}
\def\Nstop{{\mathcal{N}_{\text{stop}}}}
\def\Nstop{{\mathcal{N}_{\text{live+abs.}}}}

\def\Zall{{Z}_{\text{all}}}
\def\Zstop{{Z}_{\mathrm{stop}}}
\def\Zstop{{Z}_{\mathrm{live+abs.}}}

\newcommand{\TWs}{\TW_{s_0}}

\title{Branching Brownian motion with absorption and the all-time minimum of branching Brownian motion with drift}

\author{Julien Berestycki\thanks{Department of Statistics, University of Oxford, Oxford OX1~3TG, UK. Email: \texttt{julien.berestycki@stats.ox.ac.uk}}, \'Eric Brunet\thanks{LPS-ENS, UPMC, CNRS, 24 rue Lhomond, 75231 Paris Cedex 05, France. Email: \texttt{Eric.Brunet@lps.ens.fr}},  Simon C. Harris\thanks{Department of Mathematical Sciences, University of Bath, Bath BA2~7AY, UK. Email: \texttt{S.C.Harris@bath.ac.uk}}, Piotr Mi\l{}o\'{s}\thanks{ Faculty of Mathematics, Informatics and Mechanics, University of Warsaw, Banacha 2, 02-097 Warszawa, Poland. Email: \texttt{pmilos@mimuw.edu.pl} } }

\begin{document}

\maketitle

\begin{abstract}
We study a dyadic branching Brownian motion on the real line with
absorption at 0, drift $\mu \in \R$ and started from a single
particle at position $x>0.$ When $\mu$ is large enough so that the
process has a positive probability of survival, we consider $\Nhits(t),$ the number of
individuals absorbed at 0 by time $t$ and for $s\ge 0$ the functions $\TW_s(x):= \E^x[s^{\Nhits(\infty)}].$  We show that $\TW_s<\infty$ if and only of $s\in[0,s_0]$ for some $s_0>1$ and we study the properties of these functions. Furthermore, for $s=0, \TW(x) := \TW_0(x) =\P^x(\Nhits(\infty)=0)$ is the cumulative distribution function of the all time minimum of the branching Brownian motion with drift started at 0 without absorption. 

We give three descriptions of the family $\TW_s, s\in [0,s_0]$ through a single pair of functions, as the two extremal solutions of the Kolmogorov-Petrovskii-Piskunov (KPP) traveling wave equation on the half-line, through a martingale representation and as an explicit series expansion. 
We also obtain a precise result concerning the tail behavior of $\Nhits(\infty)$. In addition, in the regime where $\Nhits(\infty)>0$ almost surely, we show that $u(x,t) := \P^x(\Nhits(t)=0)$ suitably centered  converges to the KPP critical travelling wave on the whole real line.

\end{abstract}
\maketitle

\section{Introduction}

Consider a branching Brownian motion in which particles move according to
a Brownian motion with drift $\mu \in \R$ and split into two particles at
rate $\beta$ independently one from another. Call $\Nall(t)$ the
population of all particles at time $t$ and call $X_u(t)$  the position of a
given particle $u \in \Nall(t)$. When we start with a single particle at
position $x$ we write $\P^x$ for the law of this process. 

In a seminal paper, \cite{kesten78}, Kesten considered the branching
Brownian motion with absorption, {\it i.e.}\@ the model just described
with the additional property  that particles  entering the negative
half-line $(-\infty,0]$ are immediately absorbed and removed. We write
$\Nabs(t)$ for the set of particles alive (not absorbed) in the branching
Brownian motion with  absorption and $\Nhits(t)$ the number of particles
that have been absorbed up to time $t$. The system with absorption is
said to become extinct if $\exists t \ge 0 : \Nabs(t)=\emptyset$ and to
survive otherwise.
We let $\Nhits(\infty):=\lim_{t\to\infty}\Nhits(t)\in \R\cup\{\infty\}$.

Depending on the value of $\mu $ one has the following behaviours (see
Figure~\ref{F:fig 1})
\begin{itemize}
\item[Regime A:] if $\mu\le-\sqrt{2\beta}$, the drift towards origin is
so large that the system goes extinct almost surely. 
$\Nhits(\infty)$ is finite and non-zero.
\item[Regime B:] if $-\sqrt{2\beta}<\mu<\sqrt{2\beta}$ there is a
non-zero probability of survival. On survival, there will always be
particles near 0 and  $\Nhits(\infty)=\infty$ almost surely.  
\item[Regime C:] if $\mu\ge\sqrt{2\beta}$ there is still a non-zero
probability of survival, but the system is drifting so fast away from 0
that, on survival,   $ \min_{ u \in \Nall(t) } X_u(t) $ drifts to
$+\infty$ almost surely as $t \to \infty$; $\Nhits(\infty)$ is thus
almost surely finite. Furthermore, there is a non-zero probability that
$\Nhits(\infty)=0$.
\end{itemize}

\begin{figure}[ht]\label{F:fig 1}
\begin{center}
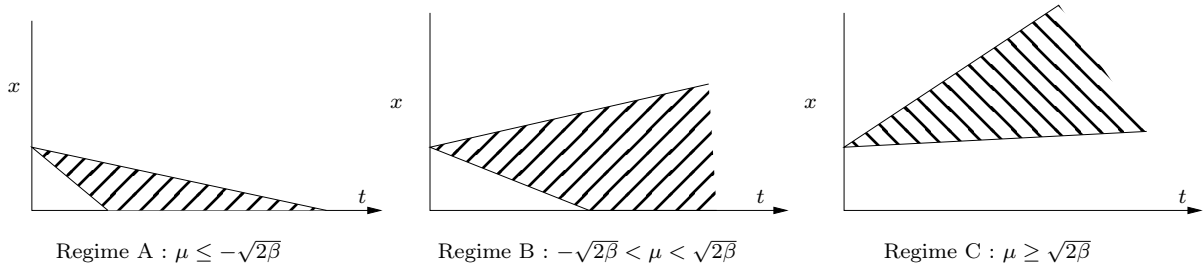
\caption{ The three regimes: The hashed region represents the cone in which
one expects to find particles.}
\end{center}
\end{figure}

The behaviour of $\Nhits(\infty)$ in regime $A$ ($\mu\le-\sqrt{2\beta}$)
has
been the subject of very active research recently, including a conjecture
by Aldous which was recently settled by P. Maillard \cite{maillard} (we
discuss Maillard's results bellow), improving earlier results of L.
Addario-Berry and N. Broutin \cite{abb} and E. A\"\i d\'ekon \cite{aidekon}. 
 Surprisingly, relatively little was
known concerning the regimes B and~C. Our main results in the present
work concern the study of $\Nhits(\infty)$ and of
certain related KPP-type equations. 

\subsection{The tail behaviour of $\Nhits(\infty)$}

In \cite{maillard}, Pascal Maillard has shown that in regime A ($\mu \le
-\sqrt{2\beta}$) the variable $\Nhits(\infty)$ has a very fat tail. More
precisely he shows that, as $z\to\infty$, there exists two constants $c,c'$ which depend on $x$ such that
$$
\P^x[\Nhits(\infty)>z] \sim 
\begin{cases}
c /(z\log (z)^2) &\text{for $\mu = -\sqrt{2\beta}$},\\[.5ex]
c' z^{-a(\mu)} &\text{for $\mu < -\sqrt{2\beta}$ where
$a(\mu) =\frac{\mu +\sqrt{\mu^2-2\beta} }{\mu -\sqrt{\mu^2-2\beta}}$.}
\end{cases}
$$

In regime B ($-\sqrt{2\beta } <\mu <\sqrt{2\beta}$) it is clear that
$\Nhits(\infty)=\infty$ on survival so one would essntially condition
on extinction to study the tail behaviour of
$\Nhits(\infty)$. This is outside the scope
of the present work.

In regime~C ($\mu \ge \sqrt{2\beta}$), however, $\Nhits(\infty)$ is
almost surely finite. We introduce for $s\ge0$ and $x\ge0$,
\begin{equation}
\label{w_s}
\TW_s(x):=\E^x\big[s^{\Nhits(\infty)}\big], \quad \TW_s(0)=s.
\end{equation}
 When $s\in[0,1]$ this quantity is the generating
function of $K(\infty)$. We show that 
$\omega_s(x)$ is finite for
some values of $s$ larger than 1:

\begin{thm}\label{main thm: exp moments}
In regime~C ($\mu \ge \sqrt{2\beta}$), there exists a finite $s_0>1$
depending only on $\mu/\sqrt\beta$ such that
\begin{enumerate}
\item For $s\le s_0$, $\TW_s(x)$ is finite for all $x\ge0$,
\item For $s >  s_0$, $\TW_s(x)$ is infinite for all $x>0$.
\end{enumerate}
The functions $x\mapsto\TW_s(x)$ are increasing for any
$s\in[0,1)$ and decreasing for any $s\in(1,s_0]$, converging to~1 when
$x\to\infty$, and one has $\TW'_{s_0}(0)=0$.

Furthermore,  one has
for $n$ large
\begin{equation}\label{maillard stuff}
\P^x[\Nhits(\infty)=n]\sim { -\TWs'(x)  \over  2s_0^{n} n^{\frac32} \sqrt{\pi\beta(s_0-1)}  }.
\end{equation}
\end{thm}


Using the branching structure and a simple coupling allows to relate the
$\TW_s(x)$ with each others, as in the following result (which is already
present in \cite{neveu,maillard}):
\begin{thm}\label{thm w_s}In regime~C ($\mu\ge\sqrt{2\beta}$),
\begin{enumerate}
\item For each $s\in [0,1)$,\qquad
$\TW_{s}(x)=\TW_0(x+\TW_0^{-1}(s))$,
\item For each $s\in(1,s_0]$,\qquad
$\TW_s(x) =\TW_{s_0}(x+\TW_{s_0}^{-1}(s)).$
\end{enumerate}
\end{thm}

We do not have an explicit expression for $s_0$ as a function of
$\mu/\sqrt\beta$, but we can evaluate it numerically with a good precision
as shown in Figure~\ref{figs0}.
In the critical case $\mu =\sqrt{2\beta} $, we obtain $s_0=1.3486\ldots$.

\begin{figure}[ht]\centering
\includegraphics[width=.6\textwidth]{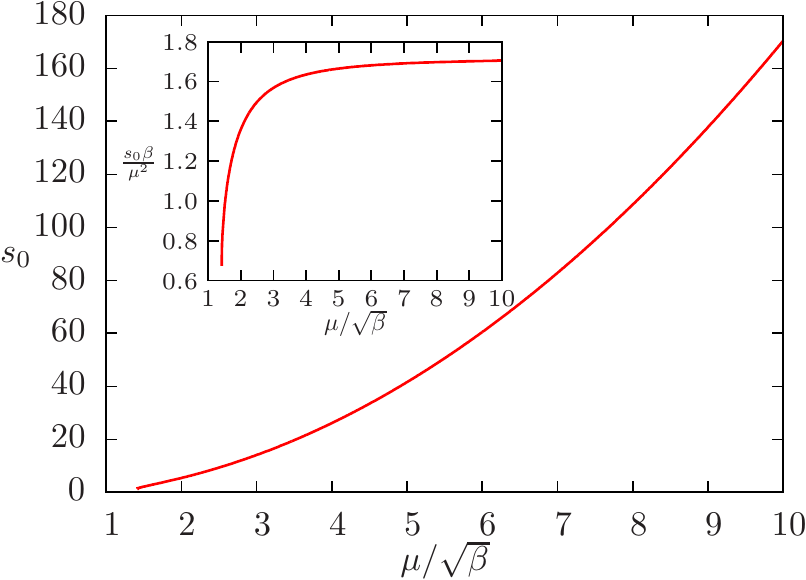}
\caption{Numerical determination of $s_0$
and $s_0\beta/\mu^2$ as functions of $\mu/\sqrt\beta$.}
\label{figs0}
\end{figure}

We can prove the following property:
\begin{prop}\label{prop:s0}
$s_0$ is an increasing function of $\mu/\sqrt\beta$ and furthermore
$s_0 \sim
c\mu^2/\beta$ for some constant~$c$ as $\mu/\sqrt\beta \to
\infty$.
\end{prop}


\subsection{Distribution of the all-time minimum in a branching Brownian motion}

The probability that $K(\infty)=0$ for a system started from $x$,
is also the probability that the all-time minimum of a full branching
Brownian motion with drift $\mu$ started from zero does not go below
$-x$:
\begin{equation}\label{def omega}
\TW(x) := \TW_0(x) 
	= \P^x[\Nhits(\infty)=0]
	= \P^0[\min_{t \ge 0} \min_{u \in \Nall(t)} X_u(t) >-x]. 
\end{equation}
This quantity, of course, is not trivial only in regime~C ($\mu\ge \sqrt{2\beta}$). Then, since 
$$ \lim_{t \to \infty} \min_{u \in \Nall(t)} X_u(t) =+\infty$$
almost surely, we see that there is a well defined all-time minimum for the branching Bownian motion and we conclude that $\lim_{x \to \infty} \TW(x)=1.$

It is not hard to see by standard arguments that $\TW$ must satisfy a
KPP-type differential equation with boundary conditions:
\begin{equation}\label{E:TW}
\begin{cases}0= \frac12 \TW''  + \mu \TW'+\beta(\TW^2-\TW), & x\ge 0, \\ 
\TW(0)=0, \ \ \TW(\infty) = 1. \end{cases}
\end{equation}
In fact, $\TW_s(x)$ introduced in \eqref{w_s}, if finte, is solution to the same
equation with the boundary condition $\TW(0)=0$ replaced by
$\TW_s(0)=s$:
\begin{equation}\label{E:TWs}
\begin{cases}0= \frac12 \TW_s''  + \mu \TW_s'+\beta(\TW_s^2-\TW_s), & x\ge 0, \\ 
\TW_s(0)=s, \ \ \TW_s(\infty) = 1. \end{cases}
\end{equation}

This is an example of the deep connection between branching Brownian
motion and the KPP equation which goes back to McKean \cite{mckean} who
noticed that one can represent solutions of the KPP
equation as expectations of functionals of branching Brownian motions. 

Until now this is very classical, however there is one unexpected
difficulty here: both \eqref{E:TW} and \eqref{E:TWs} admit infinitely
many solutions  and  are not sufficient to characterize $\TW(x)$.
Figure~\ref{F:solutions} gives
several numerical solutions to \eqref{E:TW}.
\begin{figure}[ht]\begin{center}
\includegraphics{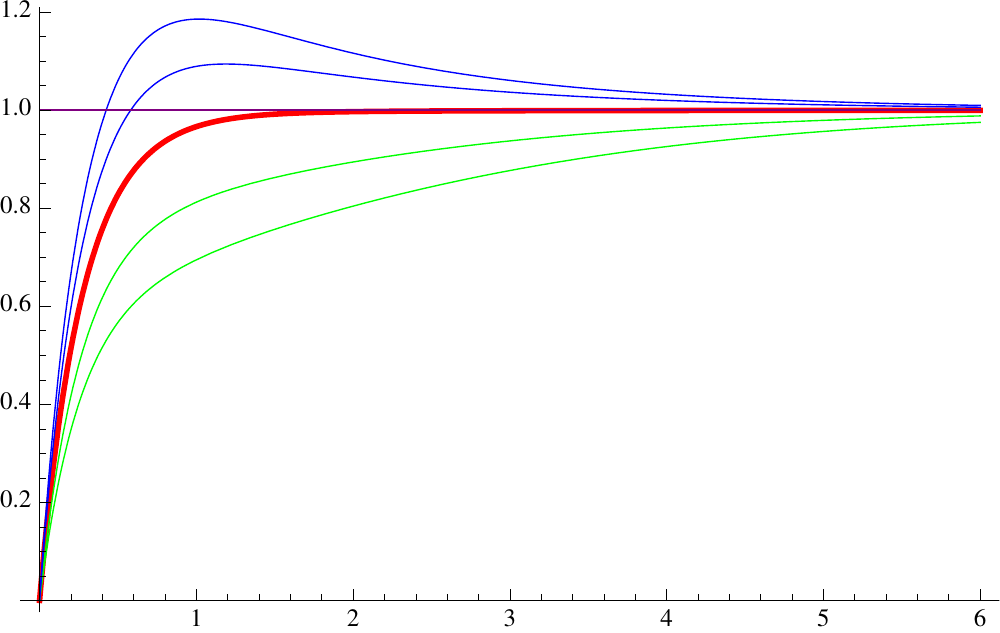}
\end{center}
\caption{Solutions to \eqref{E:TW} for $\beta=1$ and
$\mu=2$. The red (bold) curve is
$\TW$.}\label{F:solutions}
\end{figure}

In this work , we present three largely independent ways to characterize
$\TW(x)$ which are laid out in the three following subsections. The first approach
relies on partial differential equations, the second gives $\TW(x)$ as the
expectation of a certain martingale and the third one  gives
$\TW(x)$ as a power series.

One salient property of $\TW$  is that it converges to~1 rather quickly:
\begin{prop}\label{prop3}
There exists $B>0$ such that  
\begin{equation}\label{E:asymptotique}
1-\TW(x) \sim B e^{-\big(\mu +\sqrt{\mu^2-2\beta}\big) x}\qquad\text{for large
$x$}.
\end{equation}
Furthermore, $\TW(x)$ is the only solution of \eqref{E:TW} which remains in
$[0,1)$ and converges that fast to~1.

Similarly, for any $s\in[0,s_0)$, there exists $B_s \in \R$ such that 
\begin{equation}\label{E:asymptotique2}
1-\TW_s(x) \sim B_s e^{-\big(\mu +\sqrt{\mu^2-2\beta}\big) x}\qquad\text{for large
$x$}
\end{equation}
where $B_1=0,$ $B_s>0$ when $s<1$ and $B_s<0$ when $s>1$.
\end{prop}
This fast decay was actually an unexpected result for reasons explained in
Section~\ref{KPPasymp}.

To simplify notation, we call $r$ the exponential decay rate of $1-\TW(x)$
as given in \eqref{E:asymptotique}:
\begin{equation}\label{eq:r}
r:=\mu+\sqrt{\mu^2-2\beta}.
\end{equation}
It is the largest solution of $\frac12r^2-\mu r+\beta=0$.

\subsubsection{$\TW(x)$ from partial differential equations}

A first way to characterize $\TW(x)$ is to track the probability that no
particle got absorbed up to time~$t$.
Define
\begin{equation}\label{eq:u}
\pr(t,x) := \P^x[\Nhits(t)=0].
\end{equation}
The function $\pr : \R_+^2 \mapsto [0,1]$ is increasing in $x$ and
decreasing in $t$ and, clearly, for each $x$, $\pr(t,x)\to\TW(x)$ as
$t\to\infty$. Furthermore, $\pr$ satisfies the KPP equation with boundary conditions 
\begin{equation}
\label{KPP}
 \begin{cases}  \partial_t\pr = \frac12 \partial_{xx}\pr  + \mu \partial_x\pr +\beta(\pr^2-\pr), \\ \pr(t,0)=0 \ (\forall t \ge 0) ,\quad \pr(0,x) = 1  \ (\forall x>0), \end{cases}
\end{equation}
which, by Cauchy's Theorem has only one solution. 

Therefore, to obtain $\TW(x)$, one can  in principle solve \eqref{KPP} and
take the large time limit. This route leads to the following Theorem:
\begin{thm}\label{thm2}
In regime~C ($\mu \ge \sqrt{2\beta}$), the function $\TW(x)$ defined by
\eqref{def omega} is the maximal solution of~\eqref{E:TW} such that
$\TW(x) <1$ for all $x\ge 0$.

More generally, $\omega_s(x)$ for $s<1$ is the maximal solution of
\eqref{E:TWs} that stays below 1 and $\omega_s(x)$ for $s\in(1,s_0]$ is
the minimal solution of \eqref{E:TWs} that stays above~1.
\end{thm}

\subsubsection{Martingale representation for $\TW$}
\label{S: mg representation}
There is an explicit probabilistic representation of the maximum standing
wave $\TW$ in regime~C ($\mu \ge \sqrt{2\beta}$). Recall that $\Nall(t)$
is the population of all the particles in  the branching Brownian motion
with no absorption and  $\Nabs(t)$ is the population of particles alive at
time $t$ when we kill at 0. We now define on the same probability space
a third process based on the branching Brownian motion in which particles
that hit 0 are stopped but not removed from the system (they neither move
nor branch). We
denote by  $\Nstop(t)$ the set of particles alive at time $t$ in this
model. With a slight abuse of notations we continue to write $X_u(t)$ for
the positions of particles when $u \in \Nstop(t).$ 

Let us define the following two processes:
\begin{equation}\label{def of Z}
\Zstop(t) := \sum_{u \in  \Nstop(t)} e^{-r X_u(t)} ,
 \qquad \Zall(t) := \sum_{u \in  \Nall(t)}e^{-r X_u(t)},
\end{equation}
where $r$ is the asymptotic decay of $\TW(x)$ as given in~\eqref{eq:r}.
Rewriting $X_u(t) = Y_u(t) +\mu t$ it is clear that  $\big\{Y_u(t), u \in
\Nall(t)\big\}$ is simply a standard branching Brownian motion with no drift.
Therefore  $\Zall$ is the usual exponential
martingale  with parameter $r$ associated with the branching Brownian
motion $Y.$ The process $\Zstop$ is the martingale $\Zall$ stopped on the
line $t\wedge T_0$ (i.e.\@ particles are stopped at time $t$ or when they hit
0 for the first time). It is therefore also a martingale. 

\begin{lem}\label{L:conv of Z}
In regime~C ($\mu  \ge \sqrt{2\beta}$) the martingale $(\Zstop(t), t\ge 0)$
converges almost surely and in $L^1$ to $\Nhits(\infty)$ and therefore 
 $\E^x[\Nhits(\infty)] =\Zstop(0)=e^{-rx}.$
\end{lem}

Following the probability tilting method pioneered by \cite{lpp} and
\cite{chauvrou} we introduce a new probability measure $\Q^x $ 
$$
\frac{\diffd \Q^x }{ \diffd \P^x}  =e^{rx} \Nhits(\infty).
$$
Note that since $\Zstop$ is a closed martingale we have that $\E^x[\Nhits(\infty) |\cF_t] =\Zstop(t).$ Thus
$$
\left.\frac{\diffd \Q^x }{ \diffd\P^x} \right|_{\cF_t} =\frac{\Zstop(t) }{
\Zstop(0)}.
$$
Under this tilted probability measure, the law of the process is
the same as the original $\P^x$ law except for the movement and branching
rate of a distinguished particle (the {\it spine} particle $\xi$). The
spine moves according to a Brownian motion with drift
$-\sqrt{\mu^2-2\beta}$, branches at an accelerated rate of $2\beta$ and
stops (i.e.\@ sticks and stops reproducing) upon hitting 0.

\begin{thm} \label{Thm: martingale representation}
In regime~C ($\mu \ge \sqrt{2\beta}$),
$$1-\TW(x) =\Q^x\bigg( \frac1{\Nhits(\infty)} \bigg) e^{-rx}.$$
Furthermore, $\Q^x\big( \frac1{\Nhits(\infty)} \big)$ converges to
a finite constant $B>0$ when $x
\to \infty$ and thus, as $x\to \infty$,
$$
1-\TW(x) \sim Be^{-rx}.
$$
More generally, for any $s\in[0,s_0]$, one has
$$
1-\TW_s(x) = \Q^x\bigg(
\frac{1-s^{\Nhits(\infty)}}{\Nhits(\infty)} \bigg)e^{- rx}
$$
and the expectation $\Q^x(\cdot)$ converges to a finite positive constant as
$x\to\infty$.

\end{thm}

We will see in the proof that we can give an explicit representation of the constant $B$ which appears in Proposition \ref{prop3} as the expectation of $\Nhits(\infty)^{-1}$ under the measure $\Q^\infty$ (similar to $\Q^x$ but with the spine particle ``started at infinity''). 

\subsubsection{$\TW(x)$ as a series expansion}
\label{series}

The function $\TW(x)$ can be understood in terms of series expansion.
Let $\{a_{n}\}_{n\geq1}$ be the sequence defined by \begin{equation}
a_1=1,\qquad 
a_{n}=\frac{\beta}{\frac{1}{2}n^{2}r^{2}-n\mu r+\beta}\sum_{j=1}^{n-1}a_{j}a_{n-j}
=\frac1{(n-1)\big(\frac{r^2}{2\beta}n-1\big)}
\sum_{j=1}^{n-1}a_{j}a_{n-j}
,\quad n\geq2,\label{eq:expansionCoefficients}
\end{equation}
(recall that $r$ was defined in \eqref{eq:r}, that $\frac12r^2-\mu r+\beta=0$ and that
$r\ge\mu\ge\sqrt{2\beta}$)
and $\Phi$ the function defined by the series
\begin{equation}\label{Phi}
\Phi(z)=\sum_{n\ge1}a_n z^n.
\end{equation}
We have
\begin{prop}
\label{thm:solutionExpansion}
The radius of convergence $\mathcal{R}$ of $\Phi$ is non-zero and there
exists $B \in(0,\mathcal R)$ such that
$$\omega(x)=1- \Phi(B e^{-rx}).$$
More generally, for any $0\le s\le s_0$, there exists a number $B_s$ such
that
\begin{equation}\label{omegaphi}
\omega_s(x)=1- \Phi(B_s e^{-rx})\qquad\text{for all $x\ge0$ such that
$|B_s|e^{-rx}<\mathcal R$}.
\end{equation}
$s\mapsto B_s$ is decreasing, positive for $s<1$, zero for
$s=1$, and negative for $s>1$. In particular, for $s\le1$, the condition
$|B_s|e^{-rx}<\mathcal R$ is automatically fulfilled.
\end{prop}
The proof is contained in Section \ref{sub:Expansion-formulaProof}.
Numerically, it seems that $\mathcal R$ is large enough that
$|B_s|e^{-rx}<\mathcal R$ for all $s\in[0,s_0]$ and all $x\ge0$, but we
haven't proved that point.
The representation~\eqref{omegaphi} makes it very easy to compute
numerically $\TW_s$ by first computing $\Phi(z)$, see
Figure~\ref{figPhi}; the value $s_0$ is then obtained as 1 minus the
first minimum of $\Phi$ for negative  arguments. This follows easily from the facts that $w'_{s_0}=0, w''_{s_0}<0$ and that $w'_s(0)<0$ for all $s\in (1,s_0)$.  

\begin{figure}[ht]
\centering
\includegraphics[width=.45\textwidth]{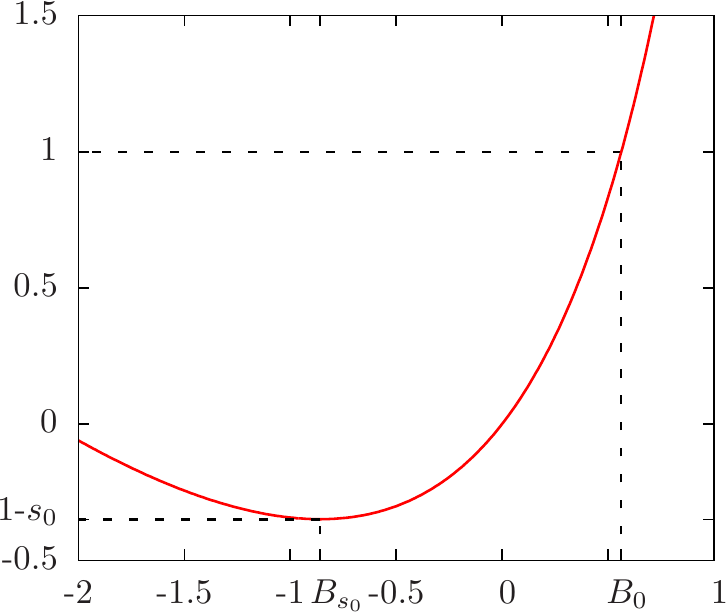},\ 
\includegraphics[width=.45\textwidth]{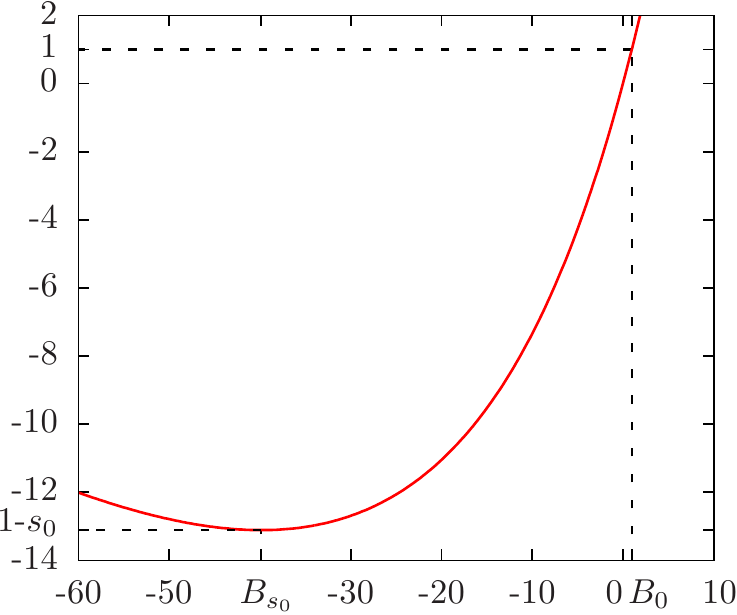}
\caption{The function $z\mapsto\Phi(z)$ for $\beta=1$ and
$\mu=\sqrt{2}$ (left) or $\mu=3$ (right). Numerically, one has
$s_0\approx 1.35$, $B_{s_0}\approx-0.859$, $B_0\approx0.564$,
$\mathcal R\approx 3.14$ for
$\mu=\sqrt2$ and
$s_0\approx14.11$, $B_{s_0}\approx-39.86$, $B_0\approx0.969$, $\mathcal
R\approx72.8$ for $\mu=3$.}
\label{figPhi}
\end{figure}

\subsubsection{About the asymptotic decay \eqref{E:asymptotique}}
\label{KPPasymp}

The KPP partial differential equation, 
\begin{equation}
\partial_t h = \frac12 \partial_{xx} h  + \beta(h^2-h),
\label{reallyKPP1}
\end{equation}
where $0\le h(t,x)\le 1$, $h(t,-\infty)=0$ and $h(t,+\infty)=1$,
describes how a stable phase ($h=0$ on the left) invades an unstable phase
($h=1$ on the right). It is well known that it admits travelling wave
solutions of the form 
$$
h(t,x)=h_\mu(x-\mu t),\qquad 0<h_\mu<1,\qquad
h_\mu(+\infty)=1, \qquad h_\mu(-\infty)=0,
$$
for any velocity $\mu$ greater or equal to $\sqrt{2\beta}$. The travelling
wave $x\mapsto h_\mu(x)$ is then solution to
\begin{equation}
\frac12h_\mu''+\mu h_\mu' +\beta(h_\mu^2-h_\mu)=0
,\quad
h_\mu(-\infty)=0
,\quad
h_\mu(+\infty)=1
.
\label{defhmu}
\end{equation}
The solution to \eqref{defhmu} is unique up to translation.

Equation~\eqref{defhmu} for $h_\mu$ is very similar to equation
\eqref{E:TW} for $\TW(x)$, but $h_\mu$, surprisingly, does not have the
same asymptotic behaviour as $\TW$ for large~$x$, in the region where
$h_\mu$ and $\TW$ are close to~1. Indeed,
linearising \eqref{defhmu} around~1, one gets
\begin{equation}\label{linearized}
\frac12(1-\tilde h_\mu)''+\mu (1-\tilde h_\mu)' +\beta(1-\tilde h_\mu)=0,
\quad\text{[linearized]}
\end{equation}
(a term of order $(1-\tilde h_\mu)^2$ has been neglected) and the general solution
to \eqref{linearized} is, for some constants $A$ and $B$,
\begin{equation}
\label{sollinearized}
1-\tilde h_\mu(x)=
\begin{cases}
 A e^{-\big(\mu -\sqrt{\mu^2-2\beta}\big) x}
+B e^{-\big(\mu +\sqrt{\mu^2-2\beta}\big) x}
&\text{for $\mu>\sqrt{2\beta}$},\\
(Ax+B)e^{-\sqrt{2\beta}\,x}
&\text{for $\mu=\sqrt{2\beta}$}.\\
\end{cases}\quad\text{[linearized]}
\end{equation}
For $x$ large, $\tilde h_\mu$ is close to $h_\mu$ with the
meaning that for some constant $A$ and $B$, $1-\tilde h_\mu\sim1-h_\mu$
(see the discussion in Section~\ref{proofprop3}).
Of course, if $A\ne0$, the term in factor of $B$ is negligible compared to
the term in factor of $A$ and the $A$ term alone is an equivalent to
$1-h_\mu$. When solving \eqref{defhmu}, it turns out that the solution has
a non-zero $A$ term and that, therefore, 
\begin{equation}
1-h_\mu(x)\sim \begin{cases}
A e^{-\big(\mu
-\sqrt{\mu^2-2\beta}\big) x}&\text{for $\mu>\sqrt{2\beta}$},\\
A x e^{-\sqrt{2\beta x}} &\text{for $\mu=\sqrt{2\beta}$},
\end{cases}
\label{decayTW}
\end{equation}
where $A$ depends on $\mu$.

We now consider equation~\eqref{E:TW} for $\TW(x)$.
Of course, the boundary condition of~\eqref{E:TW} is not sufficient to
determine a unique solution, and for a  range of values of $c$ there
exists a solution to
\begin{equation}\label{eq1-hmu}
0=\frac12 v''+\mu  v' +\beta(v^2-v),\quad
v(+\infty)=1,\quad
v(0)=0,\quad v'(0)=c.
\end{equation}
(The difference with equation \eqref{E:TW} for $\TW(x)$ is the added
condition $v'(0)=c$. Figure~\ref{F:solutions} shows several
solutions.)

One can then do, as above, a large~$x$ analysis of $v$ and, the
partial differential equation being the same, one finds again that
$1-v\sim 1-\tilde h_\mu$ ($x$ large) as given in
\eqref{sollinearized} for some $c$ and $\mu$ dependent values of $A$ and $B$.
Generically, $A$ is non-zero and $1-v$ decays as
$1-h_\mu$ in \eqref{decayTW} (up to a multiplicative constant; the $A$
is usually different and can even be negative, see Figure~\ref{F:solutions}).

However, for a well chosen value of $c$ (depending on $\mu$), one has
$A=0$, and the asymptotic decay of $1-v$ is given by the $B$ term (that
is: it decays much faster). The meaning of Proposition~\ref{prop3} is
that $\TW(x)$ is precisely that very special solution to~\eqref{E:TW} that
decays unlike all the other ones and unlike the travelling wave $h_\mu$.

\bigbreak

It is also interesting to remark that an equation very similar
to~\eqref{E:TW} appears in the study \cite{HHK,bbs2010} of the the extinction probability of
a branching Brownian motion with absorption  and supercritical drift $\mu>-\sqrt{2\beta}$ 
(regimes B and C): let
$\theta(x)=\P^x[\Nabs(\infty)=\emptyset]$ be the
extinction probability when the system is started from~$x$; then one has
\begin{equation}\label{E:extinct}
\begin{cases} 0=\frac12 \theta''  + \mu \theta'+\beta(\theta^2-\theta),
& x\ge 0, \\ 
 \theta(0)=1, \ \theta(\infty) = 0.  \end{cases}
\end{equation}
Equations \eqref{E:TW} and \eqref{E:extinct}  differ only by their boundary
conditions; however \eqref{E:extinct} has a unique solution, whereas
\eqref{E:TW} has many.

A possible way to understand the difference is that an asymptotic analysis
of $\theta(x)$ for large~$x$ similar to \eqref{sollinearized} yields only
one possible exponential decay: $\theta(x)\sim
A \exp[-(\mu+\sqrt{\mu^2+2\beta})x]$ for some constant~$A$, which means
that, up to translations, there is only one solution which does converge to zero
at infinity, whereas for $\TW(x)$ there
were two possible exponential decays and infinitely many solutions.
Otherwise said, if one were to impose $\theta(0)=1$ and $\theta'(0)=-c$,
there would only be one value of $c$ for which $\theta$ would converge to
zero at infinity.

\subsection{A travelling wave result}

In regimes A and~B ($\mu<\sqrt{2\beta}$) one has
$\TW(x)=0$ because $\Nhits(\infty)>0$ almost surely, 
which is not very interesting. What is more interesting is the way
that
$\pr(t,x)$, defined in \eqref{eq:u} as $\P^x\big[K(t)=0\big]$, converges to
zero: it does so by assuming the shape of the
critical travelling wave of the KPP equation.
Let us recall
quickly the well known facts on this critical travelling wave.

Consider the KPP equation~\eqref{reallyKPP1} without drift on the whole
line with Heaviside initial conditions:
\begin{equation}
\label{reallyKPP}
\begin{cases}
\partial_t h = \frac12 \partial_{xx} h  + \beta(h^2-h),
\\
h(0,x)=0 \ (\forall x < 0) ,\qquad h(0,x) = 1  \ (\forall x>0).
\end{cases}
\end{equation}
It is well known that $h(t,x)$ is the probability that the leftmost
particle at time $t$ of a branching Brownian motion started at $x$ is to the right of zero.
Furthermore, this probability converges to the critical travelling
wave in the following sense:
$$h(t,m_t+x)\to h_*(x) \text{ uniformly in $x$ as $t\to\infty$}$$
with
\begin{equation}
\label{m_t}
m_t :=\sqrt{2\beta}\,t-\frac{3}{2\sqrt{2\beta}}\log t +\text{Cste}
\end{equation}
and where $h_*:=h_{\sqrt{2\beta}}$ is the travelling wave moving at the
minimal possible velocity $\sqrt{2\beta}$, see~\eqref{defhmu}. To fix the
invariance by translation, we impose the further condition $h_*(0)=1/2$:
\begin{equation} \label{h_*}
\begin{cases} 
0= \frac12 h_*''  + \sqrt{2\beta}\, h_*'+\beta(h_*^2-h_*),  \\
h_*(-\infty) = 0  ,\quad
h_*(0)=\frac12
,\quad
h_*(+\infty)=1
, \end{cases}
\end{equation}
and the solution to~\eqref{h_*} is now unique.

Adding a drift $+\mu\partial_x h$ to \eqref{reallyKPP} would only shift the
solutions by $-\mu t$ and would make \eqref{reallyKPP} very similar to
\eqref{KPP}: the only difference would be that $\pr$ is defined on $\R^+$
and $h$ on $\R$, but as both equations converge quickly to zero around the
origin in regimes~A and~B $(\mu <\sqrt{2\beta})$, this difference turns out
to be minimal
and one has
\begin{thm}
\label{thm:KPP}
In regimes~A and~B ($\mu<\sqrt{2\beta}$), there
exists a constant~$C$ depending on $\mu$ and $\beta$ such that
\[
\pr(t,x+m_t-\mu t+C) \to h_*(x) \quad \text{uniformly in } x \text{ as } t \to \infty
\]
where $m_t$ is given by \eqref{m_t} 
and $h_*$ is the solution to \eqref{h_*}.
\end{thm}

It is interesting to compare this result about the behaviour of $u(t,x)
= \P^x(\Nhits(t)=0)$ when $\mu <\sqrt{2\beta}$ to the behaviour of the
extinction probability $ \tpr(t,x) = \P^x(\Nabs(t)=\emptyset )$ when $\mu \le
-\sqrt{2\beta}.$ It is not hard to see that 
$\tpr$ satisfies the same equation \eqref{KPP} as $u$ with different boundary conditions, which is $1$ minus the boundary condition in \eqref{KPP}; namely $\tpr$ solves 
\begin{equation}
\label{KPP tilde}
 \begin{cases}  \partial_t  \tpr = \frac12 \partial_{xx}\tpr  + \mu
\partial_x\tpr +\beta(\tpr^2-\tpr), \\ \tpr(t,0)=1 \ (\forall t \ge 0)
,\quad \pr(0,x) = 0  \ (\forall x>0), \end{cases}
\end{equation}
What is
particularly striking is that in the critical case $\mu=-\sqrt{2\beta}$ it
is known since Kesten \cite{kesten78} (see also \cite{bbscritical}) that
to survive up to time $t$ one must start with an initial particle at
position $x=c t^{1/3}$. This means that if $\tpr(t, x+\tilde m_t)$
converges to some limit front shape then the centering term giving the
position of the front $\tilde m_t$ has to be of order $c t^{1/3}$. However
the convergence of the solution of \eqref{KPP tilde} to a travelling wave is
at present an open problem. 

\section{Proofs}

The proofs section are presented mostly in the same order as the results:
Section~\ref{proofsKinfty} contains the proofs to Theorems~\ref{main thm:
exp moments} and~\ref{thm w_s} about the
structure of the functions $\TW_s(x)=\E^x[s^{\Nhits(\infty)}]$ in
regime~C ($\mu\ge\sqrt{2\beta}$).
Section~\ref{Proofsomega} contains the proofs to Proposition~\ref{prop3},
Theorem~\ref{thm2}, Lemma~\ref{L:conv of Z}, Theorem~\ref{Thm: martingale
representation} and Proposition~\ref{thm:solutionExpansion} about the
different representations of $\TW(x)=\TW_0(x)=\P^x[\Nhits(\infty)=0]$
in regime~C ($\mu\ge\sqrt{2\beta}$). Finally, Section~\ref{S:proof mu<muc}
contains the proof of Theorem~\ref{thm:KPP} about the establishment of a
travelling for $\pr(t,x)=\P^x[\Nhits(t)=0]$ in regimes~A and~B
($\mu<\sqrt{2\beta}$) and  the asymptotic behavior of $s_0$ established in Proposition~\ref{prop:s0}  is proven last.

\subsection{Proof concerning the tail behaviour of $\Nhits(\infty)$}
\label{proofsKinfty}

In this section we consider exclusively regime~C ($\mu\ge\sqrt{2\beta}$) and
we focus on the problem of the exponential moments of $\Nhits(\infty)$.
We first establish some properties of $\TW_{s}(x)
= \E^x[s^{\Nhits(\infty)}]$ as defined in \eqref{w_s} and proceed to prove
Theorems~\ref{thm w_s} and Proposition~\ref{prop:s0}. We then prove
the asymptotic behaviour~\eqref{asymptNhits} to complete the proof of
Theorem~\ref{main thm: exp moments}.

The first property we need is that for a given~$s$, the quantity $\TW_s(x)$
is either finite for all~$x>0$ or infinite for all $x>0$:
\begin{lem}
For a given~$s$,
\[
\big(\exists x>0 \ : \ \TW_{s}(x)<+\infty\big)
\Leftrightarrow
\big(\forall x>0 \ : \ \TW_{s}(x)<+\infty\big).
\]
\end{lem}
\begin{proof}
Fix $s>0$, $x>0$ and $y>0$.
There is a positive probability, which we note $\epsilon(x,y)$, that the initial
particle starting from $x$ reaches position $y$
before any branching or killing happens. Then
$$
\TW_s(x) =\E^x[s^{\Nhits(\infty)}] \ge
 \epsilon(x,y)\E^y[s^{\Nhits(\infty)}] 
=\epsilon(x,y)\TW_s(y).
$$
Therefore, if $\TW_s(x)$ is finite, then $\TW_s(y)$ is also finite.
\end{proof}
Remark: in the following, we write $\TW_s<\infty$ when the conditions of the
lemma are met. Clearly, this is the case when $s\le1$. Furthermore, as
$s\mapsto\TW_s(x)$ is obviously increasing, if $\TW_{s_0}<\infty$ for some
$s_0$,
then $\TW_s<\infty$ for all $s<s_0$.

When $\TW_s<\infty$, it is clear by standard arguments that $\TW_s(x)$ is
solution to
\begin{equation}  \label{Eq:TW_s}
\begin{cases}0= \frac12 \TW_s''  + \mu \TW_s'+\beta(\TW_s^2-\TW_s), \\ 
\TW_s(0)=s. \end{cases}
\end{equation}

Let us now prove Theorem~\ref{thm w_s}. A slightly more general result is
given by the following Lemma:
\begin{lem}\label{lem:coupling}
If $\TW_s<\infty$
 then, for any $x\ge0$ and $h\ge0$,
$$
\TW_s(x+h) = \TW_{\TW_s(h)}(x).
$$
\end{lem}
Setting $s=0$ and renaming $\TW_0(h)$ as $s$ gives the first line of
the Theorem. Once we have proved that $s_0$ exists, setting $s=s_0$  and
renaming $\TW_{s_0}(h)$ as $s$ gives the second line of the Theorem.
\begin{proof}
Instead of starting our branching process at position $x$ and
killing particles at $0$, it is here more convenient to think of the
process as started at 0 and particles being absorbed at $-x$. This allows
to couple different values of the killing position. In particular, if
$\mathcal{H}_x$ designates the particles stopped when they first hit $-x$
and $\Nhits_x(\infty)$ is the number of particles in $\mathcal{H}_x$, we
have that
$$
\TW_{s}(x+h) 
=\E\big[ \prod_{u \in \mathcal{H}_x} s^{\Nhits^{(u)}_h(\infty)}\big] 
=\E\big[ \prod_{u \in \mathcal{H}_x} \TW_s(h)\big] 
=\E\big[ \TW_s(h)^{\Nhits_x(\infty)}\big] 
= \TW_{\TW_s(h)}(x),
$$
where $\Nhits^{(u)}_h(\infty)$ is the total number of descendent of the particle $u$ which are killed at $-x-h$ (which by translation invariance of the branching Brownian motion and the branching property is an independent copy of $\Nhits_h(\infty)$).
\end{proof}

We now have a monotonicity result:
\begin{lem}\label{mono}\ 
\begin{enumerate}
\item If $s<1$, $x\mapsto\TW_s(x)$ is an increasing function converging to~1.
\item If $s>1$ and $\TW_s<\infty$, $x\mapsto\TW_s(x)$ is a decreasing function
converging to~1.
\end{enumerate}
\end{lem}
\begin{proof}
Once the increasing/decreasing part is proved, the fact that the limit is~1
is obvious: from its definition, it is clear that $\TW_s<1$ if $s<1$ and
$\TW_s>1$ if $s>1$.
Assuming $\TW_s$  is increasing or decreasing (depending on $s$), it must have
a limit, and from~\eqref{Eq:TW_s} that limit must be 1.

From its interpretation as the distribution of the all-time minimum of
a branching Brownian motion, see \eqref{def omega},
it is furthermore clear that $\TW_0=\TW$ is an increasing function.
Then, the coupling provided by Lemma~\ref{lem:coupling} (or more simply
Theorem~\ref{thm w_s}) implies that $\TW_s$ is an increasing function for
all $s<1$.

Therefore, it only remains to prove that for $s>1$, $\TW_s$ is decreasing
when it is finite.

Assume $s>1$ and $\TW_s<\infty$. We first show that $\TW_s$ is monotonous
by considering two cases:
\begin{itemize}
\item If $\TW_s'(0)>0$ then, for all $h>0$ small enough, $\TW_s(h)>s$. But, for 
$x$ fixed, $s\mapsto\TW_s(x)$ is a strictly increasing function so
$\TW_{\TW_s(h)}(x)>\TW_s(x)$. Then by Lemma~\ref{lem:coupling},
$\TW_s(x+h)>\TW_s(x)$ for all $x$ and all $h>0$ small enough: $\TW_s$ is
increasing.
\item If $\TW_s'(0)\le0$ then, for all $h>0$ small enough, $\TW_s(h)<s$ because in
the limit case $\TW_s'(0)=0$, one has $\TW_s''(0)<0$ from \eqref{Eq:TW_s}.
Then, as in previous case, $\TW_s(x+h)=\TW_{\TW_s(h)}(x)<\TW_s(x)$ for all
$x$ and all $h>0$ small enough: $\TW_s$ is decreasing.
\end{itemize}
It now remains to rule out the possibility that $\TW_s$ is increasing for
$s>1$. Imagine that $s>1$ and $\TW_s$ increases. Then, from \eqref{Eq:TW_s},
$\TW_s''(x)\le-2\beta(\omega_s^2(x)-\omega_s(x))\le-2\beta(s^2-s)$ and
$\TW_s'(x)\le \TW_s'(0)-2\beta(s^2-s)x$, which becomes negative for $x$
large enough, in
contradiction with the fact that $\TW_s$ increases. So $\TW_s$ must
decrease for $s>1$.
\end{proof}

We need now to characterize the values of $s$ for which $\TW_s<\infty$.
\begin{lem}
Assume $s>1$. If there exists a function $v$ which solves
\begin{equation}\label{eqv}
\begin{cases}0= \frac12 v''  + \mu v'+\beta(v^2-v), \\ 
v(0)=s,\qquad v(x)\ge1 \ (\forall x\ge0), \end{cases}
\end{equation}
then $\TW_s<\infty$.
\label{lem14}
\end{lem}
Remark: The converse is obvious: when $\TW_s$ is finite, it is one of the
solutions to \eqref{eqv}. This Lemma allows to define
\begin{equation}
\begin{aligned}
s_0
&= 
	\sup\big\{s\geq1:\TW_{s}<\infty\big\},
\\&=
	\sup\big\{s\geq1:\text{a solution to \eqref{eqv} exists}\},
\end{aligned}
\label{defs0}
\end{equation}
and because $s\mapsto\TW_s(x)$ increases, one has $\TW_s<\infty$ for all
$s<s_0$.
\begin{proof}
We present two proofs: one probabilistic and one analytical.

Choose $s>1$ such that \eqref{eqv} has a solution~$v$.
We introduce the process 
$$
M_t:= \prod_{u \in  \Nstop(t)} v\big(X_u(t)\big),
$$
where we recall that $\Nstop(t)$ is the set of particles in the branching
Brownian motion where particles are frozen at the origin, see
Section~\ref{S: mg representation}.

$M_t$
is a positive local martingale and therefore a positive super-martingale
which thus converges almost surely to $M_\infty$. Observe that under
$\P^x$
$$ v(x) =M_0 \ge \E^x(M_t) \ge \E^x(M_\infty).$$ But since for all $t\ge 0$ one has 
$$M_t \ge v(0)^{\Nhits(t)}=s^{\Nhits(t)},$$
we see that $M_\infty \ge s^{\Nhits(\infty)}$ and therefore 
$$
\TW_s(x)=\E^x\big[s^{\Nhits(\infty)}\big] \le v(x) <\infty.
$$

The same result can be proved analytically through the maximum
principle. Let us introduce
\begin{equation} 
\pr_s(t,x) := \E^x[s^{\Nhits(t)}],
\end{equation}
which is clearly solution to
\begin{equation}\label{ustx}
 \begin{cases}  \partial_t\pr_s = \frac12 \partial_{xx}\pr_s  + \mu
\partial_x\pr_s +\beta(\pr_s^2-\pr_s), \qquad x\ge0,
\\ \pr_s(t,0)=s \ (\forall t \ge 0)
,\qquad \pr_s(0,x) = 1  \ (\forall x>0). \end{cases} \end{equation}
(Compare to~\eqref{eq:u}.)
With $v$ as above, one clearly has  
$\forall x\ge0$,  $\pr_s(0,x) \le v(x)$.
Therefore, by the maximum principle we have that 
$$
\pr_s(t,x) \le v(x) ,\qquad \forall t\ge 0, x\ge0
$$
and as $\pr_s(t,x) \nearrow \TW_s(x)$ as $t\to \infty$  we see that
$\TW_s(x)\le v(x)<\infty$.
\end{proof}

It is obvious that $s_0$ defined in \eqref{defs0} depends only on the ratio
$\mu/\sqrt\beta$ by a simple scaling argument: the branching Brownian
motion with drift $\mu$ and branching rate $\beta$
is transformed, when time is scaled by $\lambda$ and space by $\sqrt \lambda$, into a branching Brownian motion with drift $\mu \sqrt \lambda$ and branching rate $\beta \lambda$.
In particular,
$\TW_{s,\beta,\mu}(x)=\TW_{s,\beta\lambda,\mu\sqrt\lambda}(\sqrt\lambda\, x)=\TW_{s,1,\mu/\sqrt{\beta}}(x/\sqrt{\beta})$
with the obvious new notation.
What remains to be shown are the
following properties of $\TW_s$: $s_0$ is finite,
$\TW_{s_0}$ is finite, $\TW_{s_0}'(0)=0$ and $s_0>1$.

\begin{lem}
$s_0<\infty$, that is: $\Nhits(\infty)$ does not have exponential
moments of all orders.
\end{lem}
\begin{proof}
For the system started from~$x>0$, consider the following family of
events for $n\in\N$:
$$\mathcal A_n=\begin{cases}
\Nhits(\infty)=n,
	\quad\text{and}\\
\text{for all integers $i\le n$, }\Nhits(i)=i,
	\quad\text{and}\\
\text{for all integers $i\le n$, there is at time~$i$ only one particle
alive and it sits in $[x,x+1]$.}
\end{cases}$$ 
In words, for each $i\in\{0,1,\ldots,n-1\}$ there is one particle alive
at time $i$, it sits in $[x,x+1]$, and during a time interval one,
this particle splits exactly once, one the offspring gets absorbed and the other is
again in $[x,x+1]$ at time $i+1$. The one particle alive at time~$n$
generates a tree drifting to infinity with no more absorbed particles.

Let $\epsilon_{y,z}\diffd z$ be the probability that a particle sitting at
$y$ has, during a time interval one, exactly one splitting event with one
offspring being absorbed and the other one ending up in $\diffd z$. Define
furthermore $$q=\min_{y\in[x,x+1]}\int_x^{x+1} \epsilon_{y,z}\diffd z.$$
$q$ is the minimal probability for a particle sitting somewhere 
in $[x,x+1]$ to have,
during a time interval one, exactly one splitting event with one
offspring being absorbed and the other one ending up in $[x,x+1]$.
It is clear that $q>0$ and that, furthermore,
$$
  \P^x(\Nhits(\infty)=n)
	> \P^x(\mathcal A_n)
	> q^n\P^x(\Nhits(\infty)=0).
$$
This implies that $\E^x[q^{-\Nhits(\infty)}]=\infty$ and that
$s_0\le 1/q<\infty$.
\end{proof}

\begin{lem} $\TW_{s_0}<\infty$.
\end{lem}
\begin{proof}
If $s_0=1$, this is trivial as $\TW_1=1$. Assume now $s_0>1$ and let
us fix $x>0$. For any $1<s<s_0$, as $\TW_s$ is decreasing, one has
$\TW_s(x)<s<s_0$. This implies that (using the monotone convergence Theorem) $\TW_{s_0}(x)=
\lim_{s\nearrow s_0}\TW_s(x)$ is
finite which entails the result.
\end{proof}

\begin{lem} $\TW_{s_0}'(0)=0$. \label{s_0flat}\end{lem}
\begin{proof}
We already know that $\TW_{s_0}'(0)\le 0$.
Assuming $\TW_{s_0}'(0)<0$, one could continue the function
$\TW_{s_0}$ to negative arguments using \eqref{eqv}
and one could find a $x_0<0$ such that
$\TW_{s_0}(x_0)>s_0$; then the function  $x\mapsto\TW_{s_0}(x_0+x)$
satisfies \eqref{eqv} with $s>s_0$, which is a contradiction.
\end{proof}

The proof that $s_0>1$ for all $\mu\ge\sqrt{2\beta}$
is divided into two steps. First, we show that
$s_0>1$ in the critical case $\mu=\sqrt{2\beta}$. Then we conclude by
proving Proposition~\ref{prop:s0}, which states that $s_0$ is an increasing
function of
$\mu/\sqrt\beta$.

\begin{lem}In the critical case $\mu=\sqrt{2\beta}$, for $s>1$ small
enough, there exists solutions to \eqref{eqv}, that is $s_0>1$.
\end{lem}
\begin{proof}
Assume $\mu=\sqrt{2\beta}$. After the change of variables 
$\ell(x) := e^{\mu x} \big(v(x)-1\big)$, \eqref{eqv} reads
\begin{equation}\label{ell equation}
\frac12 \ell'' +\beta e^{-\mu x} \ell^2 =0.
\end{equation}

Let us consider the solution to \eqref{ell equation} with
$\ell(0)=\ell'(0)=\epsilon$ for some $\epsilon>0$. We want to prove that
$\forall x, \ell(x)>0$ if $\epsilon$ is small enough.
Assume otherwise and call $x_0
= \inf\{x \ge 0 : \ell(x) =0\}$. Then, as
$\ell''(x)\le0$, we have
$\ell(x)\le\epsilon+x\epsilon$
and thus on $[0,x_0]$ (where $\ell(x)\ge0$),
$$
\ell''(x) \ge -2\beta\epsilon^2(1+ x)^2e^{-\mu x}.
$$
We conclude that
$$
\ell'(x_0) \ge\epsilon -2\beta\epsilon^2 \int_0^{x_0}(1+x)^2e^{-\mu
x}\,\diffd x
\ge \epsilon -2\beta\epsilon^2 \int_0^\infty (1+x)^2 e^{-\mu x}\,\diffd x,
$$
which is strictly positive for $\epsilon$ small enough. This contradicts
the definition of $x_0$  and thus we have found a solution of \eqref{ell
equation} such that $\ell(x)>0$ for all $x\ge 0$. Then $v(x)
= 1+ \ell(x) e^{-\mu x}$ is a solution to \eqref{eqv} started from
$s=1+\epsilon$; in other words $s_0\ge 1+\epsilon$ in the
$\mu=\sqrt{2\beta}$ case.
\end{proof}

We now proceed to prove that $s_0$ is a strictly increasing function of
$\mu/\sqrt{\beta}$.
\begin{proof}[Proof of Proposition~\ref{prop:s0}]
Let us fix $\mu\ge\sqrt{2\beta_1}>\sqrt{2\beta_2}$. One can easily
construct two branching Brownian motions with parameters $(\mu,\beta_1)$
and $(\mu,\beta_2)$ on the same probability space to realise a coupling so
that
the particles of the second one are a subset of the particles of the first
one. It is then clear that for any $s>1$ one has
\begin{equation}
\TW_{s,\beta_1,\mu}(x)\ge\TW_{s,\beta_2,\mu}(x)
\label{softmon}
\end{equation}
(with the obvious extension of notation) so that
\begin{equation}
s_0(\mu/\sqrt{\beta_1})\le s_0(\mu/\sqrt{\beta_2}).
\label{softmon2}
\end{equation}
This already gives
non-strict  monotonicity and concludes the proof that
$s_0>1$ for all $\mu$, $\beta$ with $\mu\ge\sqrt{2\beta}$.

We can now prove that the inequality~\eqref{softmon2} is strict.
Assume otherwise; one would have
$\TW_{s_0,\beta_1,\mu}(0)=
\TW_{s_0,\beta_2,\mu}(0)=s_0$ (where $s_0>1$ would be the common value),
$\TW_{s_0,\beta_1,\mu}'(0)=
\TW_{s_0,\beta_2,\mu}'(0)=0$ (from Lemma~\ref{s_0flat})
and, from \eqref{Eq:TW_s},
$\TW_{s_0,\beta_1,\mu}''(0)=-\beta_1(s_0^2-s_0)
< \TW_{s_0,\beta_2,\mu}''(0)=-\beta_2(s_0^2-s_0)$, which would imply by
Taylor expansion that for $x>0$ small enough
$\TW_{s_0,\beta_1,\mu}(x)<\TW_{s_0,\beta_2,\mu}(x)$ in contradiction
with~\eqref{softmon}.
\end{proof}

Finally, the only remaining point to complete the proof of
Theorem~\ref{main thm: exp moments} is the asymptotic
behaviour~\eqref{maillard stuff}, i.e. the assertion that for $x>0$ fixed
\begin{equation}\label{to be proven (maillard)}
p_n(x):=\P^x[\Nhits(\infty)=n]\sim { -\TWs'(x)  \over  2s_0^{n}
n^{\frac32} \sqrt{\pi\beta(s_0-1)}  }
\qquad\text{as $n\to \infty$}.
\end{equation}

Write $D(z,r)$ for the open disc of the complex plane  with center $z\in
\C$ and radius $r$. We extend the definition of $s\mapsto\TW_s(x)$ to $s\in
\C$:
\begin{equation}\label{pr again}
\TW_s(x) =\E^x[s^{\Nhits(\infty)}]=\sum_{n\ge 0} s^n p_n(x), \quad \forall
s\in\C. 
\end{equation}
This quantity is analytical on $D(0,s_0)$ because the $p_n$ in \eqref{pr again}
are positive and the first singularity on the real axis is at $s_0$.
Furthermore, it is finite on $\overline{D(0,s_0)}$ by uniform convergence
because it is finite at $s_0$.

The arguments we use are extremely close to those used by Maillard in
\cite{maillard}. The key argument, which improves on usual Tauberian
theorems, is an application of \cite[Corollary
VI.1]{flajolet} relying on the analysis of generating
functions near their singular points.  We need  to show that 
\begin{lem}\label{lemma 1 in proof of Malliard}
Fix $x>0$. There exists $r_x>0$ such that   $s\mapsto \TW_s(x)$ is analytical in $V=D(s_0,r_x) \setminus [s_0,\infty),$ and
\begin{equation}\label{flaj}
\partial_s \TW_s(x) \sim { -\TW'_{s_0} (x)   \over  2\sqrt{ \beta (s_0-s) (s_0^2-s_0) }   } \qquad \text{ as } s\to s_0, s\in  V,
\end{equation}
\end{lem}
and that
\begin{lem}\label{2nd lemma for maillard}
Fix $x>0$. There exists $\epsilon>0$ such that $s\mapsto \TW_s(x)$ is analytical on $D(0,s_0+\epsilon)\setminus [s_0,\infty)$.
\end{lem}
Then, applying Corollary VI.1 in \cite{flajolet} to $s\mapsto \partial_s
\TW_s(x)$, Lemmas~\ref{lemma 1 in proof of Malliard} and~\ref{2nd lemma for
maillard} lead to
\begin{equation}\label{expansion derivative}
(n+1) p_{n+1}(x)  \sim { -\TWs'(x)  \over  2s_0^{n+1}
n^{\frac12} \sqrt{\pi\beta(s_0-1)}  }, \quad \text{ as } n \to \infty,
\end{equation}
which obviously implies  \eqref{to be proven (maillard)}.

\begin{proof}[Proof of Lemma~\ref{lemma 1 in proof of Malliard}]
We know that $\TW'_{s_0}(0)=0$ and $ \TW''_{s_0}(0)=-2\beta (s_0^2-s_0)<0.$ Since $\TW_{s_0}$ solves the KPP traveling wave differential equation, for each $x\ge 0$ we can extend $z \mapsto \TW_{s_0}(x+z)$ analytically on a neighborhood of zero in $\C$ (see e.g. \cite{rudin}). In particular for $x=0$ we have
the following expansion:
\begin{equation}\label{complex expansion}
\TWs (z) =s_0+ { \TWs''(0)\over 2}z^2 +o(z^2) \quad \text{ as } z\to 0.
\end{equation}
The function $\TWs$ is analytic and zero is a zero of order two of
$\TWs(z)-s_0$, by Theorem 10.32 of \cite{rudin}  there exists $r_1>0$ and
a function $\psi$ analytic and invertible on $D(0,r_1)$  such that
\begin{equation}\label{Mean value thm}
\TWs(z) =s_0 +{ \TWs''(0)\over 2}\psi(z)^2.
\end{equation}
This means that 
\begin{equation}\label{j'ai plus de label}
z=\psi^{-1} \left(\sqrt{\frac{\TWs(z) - s_0}{\TWs''(0)/2}}\right)
             = \psi^{-1} \left(\sqrt{\frac{s_0 - \TWs(z)}{\beta(s_0^2-s_0)}}\right) 
\end{equation}
for any $z$ in $D(0,r_1)$ such that $\TWs(z) \not \in (s_0,\infty)$ (so that
the right-hand side is well defined and is analytic on this domain when using the standard definition of the complex  square root).

Recall from Lemma \ref{lem:coupling} that for any  non-negative real $x$ and $z$ one has
\begin{equation}\label{Lemma 11 again}
\TW_{\TWs(z)}(x) =  \TWs(z+x).
\end{equation}
Replace the $z$ in the right-hand side by its expression \eqref{j'ai plus de label} and write $\TWs(z)$ as $s$ to obtain
\begin{equation}\label{nesting conseq}
\TW_{s}(x) =  \TWs \left(\psi^{-1} \left(   \sqrt{{ s_0-s \over \beta(s_0^2-s_0)  }}\right) +x\right)
\end{equation}
for  $s \in \TWs \big( [0,r_1) \big) =(s_0- r_2,s_0]$ for some $r_2>0$. But \eqref{nesting conseq} is an equality between analytical functions as long as $s \in D(s_0,r_x) \setminus [s_0,
\infty)$ for some $r_x>0$ small enough (one must have $D(s_0,r_x) \subset
\TWs\big( D(0,r_1)\big)$ for $\psi^{-1}$ to be analytical, which is possible
by the open mapping Theorem, and one must have $ \psi^{-1}(\ldots)$ small
enough for $\TW_{s_0}$ to be also analytical). From the analytical
continuation principle,  \eqref{nesting conseq} must hold on the whole
$D(s_0,r_x) \setminus [s_0,\infty)$ domain. Now differentiate with respect to
$s$ to get
\begin{equation}\label{key point}
\partial_s \TW_s(x) = -{ (\psi^{-1})'  \left( \sqrt{{ s_0-s \over \beta (s_0^2-s_0)}} \right) \over 2\sqrt{\beta(s_0-s)(s_0^2-s_0)}} \TWs'\left( \psi^{-1} \left( \sqrt{ { s_0-s \over \beta(s_0^2-s_0) } }\right) +x   \right),
\end{equation}
yielding
\begin{equation}\label{key point2}
\partial_s \TW_s(x) \sim -{ (\psi^{-1})'  (0) \over 2\sqrt{\beta(s_0-s)(s_0^2-s_0)}} \TWs'  (x) \quad \text{ as } s\to s_0 \text{ in } D(s_0, r_x)\setminus [s_0,\infty).
\end{equation}
A straightforward computation shows that $(\psi^{-1})'(0)=1,$ which concludes the proof.
\end{proof}

\begin{proof}[Proof of Lemma \ref{2nd lemma for maillard}]
$\TW_s(x)$ is already analytical on $D(0,s_0)$. To prove the Lemma it is
sufficient to show that it can be analytically extended around any point
$s\in\partial D(0,s_0)\setminus\{s_0\}$. Indeed, by the finite covering
property of compacts one can then show analycity on a open containing the
compact $\partial D(0,s_0)\setminus D(s_0,r_x/2)$ with $r_x$ defined in
Lemma~\ref{lemma 1 in proof of Malliard}, and then we conclude with the
help of Lemma~\ref{lemma 1 in proof of Malliard}.

So it now remains to see why $s\mapsto \TW_s(x)$ can be analytically
extended to neighboorhoods of any $s\ne s_0$ with $|s|=s_0$.
This is essentially the content of Lemma~6.2 in \cite{maillard}.
As in Maillard, we define
\begin{equation}
a(s):= \TW'_s(0),
\end{equation}
where we recall that the prime is a derivatice with respect to~$x$. We
first show analycity of $a(s)$ on $D(0,s_0)$ by writing an integral
representation of $a(s)$: multiply \eqref{E:TWs} by
$\exp\big[(\mu-\sqrt{\mu^2+2\beta}) x\big]$
and integrate on $x\in[0,\infty)$. Integrate several times by
part to get rid of the derivatives of $\TW_s$; one is left with
\begin{equation}
a(s)=\big(\mu-\sqrt{\mu^2+2\beta}\big)s+2\beta\int_0^\infty\diffd x\,
\TW_s(x)^2e^{\big(\mu-\sqrt{\mu^2+2\beta}\big)x}.
\label{aint}
\end{equation}
For any $x\ge0$ and $s\in\overline{D(0,s_0)}$ 
one has $|\TW_s(x)|\le \TW_{s_0}(x) \le s_0$. This implies that the
convergence for $x$ close to infinity of the integral 
in~\eqref{aint} is uniform on the disk $s\in\overline{D(0,s_0)}$.
As $s\mapsto\TW_s(x)$ is analytical on $D(0,s_0)$, this is sufficient to
ensure that $s\mapsto a(s)$ is also analytical on $D(0,s_0)$.
Furthermore, notice that the series~\eqref{pr again} defining $\TW_s(x)$
converges uniformly on $s\in\overline{D(0,s_0)}$ because we know it
converges absolutely (all the $p_n(x)$ are non-negative) at $s=s_0$. 
This implies that $s\mapsto \TW_s(x)$ is continuous on
$\overline{D(0,s_0)}$ and, from the expression~\eqref{aint}, so is
$s\mapsto a(s)$ (by dominated convergence Theorem since $|\TW_s(x)|\le s_0$ on the closed disc).

We proceed to show that $a(s)$ can be extended analytically around any
point $s\ne s_0$ with $|s|= s_0$ and show that the property extends to
$\TW_s(x)$.

In Lemma~\ref{lem:coupling} we showed for any $s\in[0,s_0]$ and any $x\ge0$
and $h\ge0$ one had
\begin{equation}
\TW_s(x+h) = \TW_{\TW_s(h)}(x).
\label{lem11again}
\end{equation}
One can check that the proof of Lemma~\ref{lem:coupling} extends to
complex~$s$ so that~\eqref{lem11again} remains valid for any $s\in C$ such
that $\TW_s(x)$ is finite.

For fixed (complex) $s$, by deriving~\eqref{lem11again} with respect to $h$
and then setting $h=0$, one gets
\begin{equation}
\TW'_s(x) = a(s) \partial_s \TW_s(x).
\label{kolmo}
\end{equation}
Derive again with respect to $x$, and then set $x=0$:
\begin{equation}
\TW''_s(0) = a(s) \partial_s a(s),
\end{equation}
so that the differential equation~\eqref{Eq:TW_s} on $x\mapsto \TW_s(x)$
applied at $x=0$ gives
\begin{equation}
0=\frac12a(s)\partial_s a(s) + \mu a(s)+\beta(s^2-s),
\label{equ 3.5 maillard}
\end{equation}
which is equation~(3.5) in \cite{maillard}. 
This equation is valid for all $s\in D(0,s_0)$.


We now use the Fact~5.2
in~\cite{maillard}, which we reproduce here with some change of notation
for clarity:
\begin{quote}
\emph{Fact~5.2 in~\cite{maillard}}\quad
Let $H$ be a region in $\C$ and $s\mapsto \phi(s)$ analytic in $H$. Let $G$ be
a region in $\C^2$ such that $(\phi(s), s) \in G$ for each $s\in H$ and
suppose that there exists an analytic function $f: G\to \C$  such that 
$$
\phi'(s) =f\big(\phi(s),s\big), \quad \forall s\in H.
$$
Let $s^* \in \partial H.$ Suppose $\phi(s) $ is continuous at $s^*$ and that
$(\phi(s^*),s^*) \in G$. Then $s^*$ is a regular point of $\phi(s)$
i.e.\@ $\phi(s)$ admits an analytic continuation at $s^*$.
\end{quote}
We apply to our case with $\phi=a$,
$H=D(0,s_0)\setminus\{1\}$ and $G=\C^*\times \C$.
From~\eqref{equ 3.5 maillard}, the only candidate values of $s$ in
$D(0,s_0)$ such that $a(s)=0$ are 0 and 1, and we know that $a(0)>0$, so
the condition ``$(a(s), s) \in G$ for each $s\in H$'' is verified.
The function $f(a,s)$ is obtained
from~\eqref{equ 3.5 maillard}:
$f(a,s)= -2\mu +2(s-s^2)/a$, and is obviously analytical on~$G$;
we have already shown that $a(s)$ is
continuous on $\overline{ D(0,s_0)}$. 
Therefore, for any point $s^*\in\partial D(0,s_0)$ such that $a(s^*)\ne0$
(because we want $(a(s^*),s^*)\in G$), the function $a(s)$ admits an
analytic continuation at $s^*$. We know that $a(s_0)=0$, and we prove now
that one has $a(s^*)\ne0$ for any $s^*\in\partial D(0,s_0)\setminus\{s_0\}$,
which will conclude the proof that $a(s)$ can be analytically continued around any
point in $\partial D(0,s_0)\setminus\{s_0\}$. From~\eqref{pr again} one can
write $a(s)$ as a series:
\begin{equation}
a(s)=\sum_n p_n'(0) s^n.
\end{equation}
We know that $p_1'(0)\le0$ (because $p_1(0)=1$ and $p_1(x>0)<1$) and that
for $n\ne1$, $p_n'(0)\ge0$ (because $p_n(0)=0$ and $p_n(x>0)>0$). Since
$a(s_0)=0$, we write
\begin{equation}
\sum_{n\ne1} p_n'(0) s_0^n = -p_1'(0) s_0.
\end{equation}
All the terms on the left hand side are non-negative and infinitely many of them are non-zero since \eqref{equ 3.5 maillard} does not have polynomial solutions. Thus, for any $s^*\in\partial
D(0,s_0)\setminus\{s_0\}$ one has
\begin{equation}
\Big|\sum_{n\ne1} p_n'(0)(s^*) ^n \Big| < \sum_{n\ne1} p_n'(0) s_0^n .
\end{equation}
In particular, $a(s^*)\ne0$ because it is the sum of two terms (the
$\sum_{n\ne1}$ and the term $n=1$) with different moduli and $s\mapsto
a(s)$ can be extended analytically around $s^*$.

We now show how the analycity of $a(s)$ translates into analycity of
$\TW_s(x)$. First
derive~\eqref{lem11again} again but this time with respect to $x$, then set
$x=0$, and rename $h$ into $x$ to obtain
\begin{equation}
\TW_s'(x) = a\big[\TW_s(x)\big] = a(s) \partial_s \TW_s(x),
\label{kolmo2}
\end{equation}
where we used~\eqref{kolmo} for the second equality.

For each
given $s^*\in\partial D(0,s_0)\setminus\{s_0\}$ we consider a neibourhood
V of $s^*$ where $a(s)$ is
analytical and we apply again Fact~5.2 in
\cite{maillard} to prove that $s\mapsto \TW_s(x)$ is also analytical around
$s^*$. This time, we take $\phi(s)=\TW_s(x)$ and $f(\TW,s)=a(\TW)/a(s)$
from~\eqref{kolmo2}. We pick $H=D(0,s_0)\setminus\{1\}$ and 
$G=D(0,s_0)\times (H\cup V)$. For any
$s\in\overline{D(0,s_0)}\setminus\{s_0\}$ one has
$|\TW_s(x)|< \TW_{s_0}(x)\le s_0$ so that the condition
``$(\phi(s),s)\in G$ for each $s\in H$'' is satisfied. We have already
shown that $s\mapsto\TW_s(x)$ is continuous on $\overline{D(0,s_0)}$, so we
conclude that $s^*$ is a regular point of $s\mapsto \TW_s(x)$.

\end{proof}




\subsection{Proofs concerning $\TW(x)$}
\label{Proofsomega}

In this section we consider exclusively regime~C ($\mu\ge\sqrt{2\beta}$)
and prove our results concerning the properties of $\TW$ and $\TW_s$.

\subsubsection{Proof of Theorem~\ref{thm2}}
\label{sec:maxprinc}

We need to prove that $\TW$ is the maximal solution remaining below~1 of
the differential equation~\eqref{E:TW}. This is an elementary application
of the maximum principle again. Suppose that $v$ is any solution of
\eqref{E:TW} which stays below 1. Since $v$
is a {\it standing wave} solution of \eqref{KPP}, that is
$\tpr(t,x)=v(x)$ for all $t\ge 0$ is a solution of
$$
\partial_t \tpr  =\frac12 \partial_{xx} \tpr +\mu \partial_x \tpr+\beta(\tpr^2-\tpr),
$$
and since $v(x) <1$ for all $x\ge 0$ we have that 
$$
v(x) \le \pr(t,x), \qquad \forall t\ge 0, \forall x\ge 0
$$
where $\pr$ is the solution of \eqref{KPP}. As for each $x$ we know that
$t\mapsto \pr(t,x)  \searrow \TW(x)$ we conclude that $\TW(x)\ge v(x)$
and therefore $\TW$ is the maximal solution of \eqref{E:TW} bounded by 1. The same argument is easily generalized to the case of an arbitrary value of $s \in  [0,s_0]$.

\subsubsection{Proof of the martingale representation, Theorem~\ref{Thm:
martingale representation}}
\label{S:proof of mg representation}

We start by proving Lemma \ref{L:conv of Z}, i.e.\@ that the martingale $(\Zstop(t), t\ge 0)$ converges $\P$-almost surely and in $L^1$ to $\Nhits(\infty)$ and therefore that
$\E^x[\Nhits(\infty)] =e^{-rx}.$

\begin{proof}[Proof of Lemma~\ref{L:conv of Z}]

Recall that by \eqref{def of Z}
\begin{equation}
\Zstop(t) := \sum_{u \in  \Nstop(t)} e^{-r X_u(t)} ,
 \qquad \Zall(t) := \sum_{u \in  \Nall(t)}e^{-r X_u(t)},
\end{equation}
are positive martingales which therefore converge $\P$-almost surely to
their respective limits
$\Zstop$ and $\Zall$. Furthermore, as $\Zall(t)$ is the usual additive
martingale with parameter $r\ge\sqrt{2\beta}$, one has $\Zall=0$.
As the bounds
$$
\Nhits(t) \le \Zstop(t)  \le \Nhits (t) + \Zall(t)
$$
always hold, it is clear that $\Zstop = \Nhits(\infty) $.
The only thing left is to show that the convergence also holds in
$L^1$.

We start by recalling the description of the measure $\Q^x|_{\cF_t}$ which is defined by
$$
\frac{\diffd \Q^x }{\diffd\P^x} \Big|_{\cF_t} =\frac{\Zstop(t)}{
\Zstop(0)}.
$$
Standard arguments (see
\cite{harrisharris}) allow us to conclude that under $\Q^x$  the process
behaves as follows: for $t\ge0$, there is a distinguished line of descent (the
\textit{spine}) denoted $\xi(t)\in
\Nstop(t)$.
Under $\Q^x$ the particle $\xi$ moves according to a Brownian motion with
drift $-\sqrt{\mu^2-2\beta}$ and therefore almost surely hits 0 in finite
time; we call $\tau_\xi =\inf\{t \ge 0 : X_{\xi(t)}(t)=0\}$ the time at
which it reaches~0. For $t<\tau_\xi$, the spine branches at rate $2\beta$
creating non-spine particles which start new independent branching
Brownian motion  behaving according to the usual $\P$ law. After
$\tau_\xi$, the spine particle is frozen at zero (no motion, no branching). Observe that $\Q^x$ is actually the projection of the measure just described since under $\Q^x$ we do not know which is the spine particle $\xi.$ 

To prove the $L^1$ convergence  of $\Zstop(t)$ towards its limit $\Zstop:= \lim_t \Zstop(t)$, it is sufficient to show that
$$
\Q^x(\Zstop <\infty) =1.
$$

As the time $\tau_\xi$ at which the spine is absorbed
at 0 is $\Q$-almost surely finite, there are only finitely many
branching events from the spine $\Q$-almost surely as well.
At each of these events, a
non-spine particle $u$ starts its  own independent $\P$ branching
Brownian motion and we call
$\Nhits_u(\infty)$
the total number of particles frozen at 0 that are descended from $u$.
Let
us also call $\Zstop^{(u)}$ the analogue of the limit $\Zstop$ (but we sum only on
particles descended from $u$) and $ \Zall^{(u)}$ is the same as
$\Zstop^{(u)}$ but without any absorption or freezing at 0. 
It is clear as above that 
$$
\Nhits_u(\infty)\le \Zstop^{(u)} \le \Nhits_u(\infty) + \Zall^{(u)}.
$$
and that $\Zall^{(u)}=0$.
We conclude that $\Zstop^{(u)} =\Nhits_u(\infty) <\infty$,
$\Q$-almost surely, and finally 
$$
\Zstop = \Nhits(\infty) <\infty \quad \text{$\Q$-almost surely.} 
$$
Observe that since $\Nhits(\infty)\ge 1$ almost surely under $\Q$,
we have $\Q \sim \P$.
Thus we know that under $\P$ $
\Zstop(t) \to \Zstop = \Nhits(\infty)
$
in $L^1$. Hence, $\E^x(\Nhits(\infty) ) =\Zstop(0) =e^{-r x}.$
\end{proof}

We now move to the proof of Theorem \ref{Thm: martingale
 representation}.  
\begin{lem}
Recall $\TW(x) =\P^x(\Nhits(\infty)=0)$ and $\TW_s(x) :=
\E^x[s^{\Nhits(\infty)}]$ for $0<s\le s_0$  as usual.
Then $$1-\TW(x) =\Q^x\Big[
\frac1{\Nhits(\infty)} \Big] e^{-r x}$$
and
$$1-\TW_s(x)=\Q^x\Big[ \frac{1-s^{\Nhits(\infty)}}{\Nhits(\infty)} \Big]
e^{-r x}.$$
\label{18}
\end{lem}

\begin{proof}
As $\Nhits(\infty)>0$ $\Q^x$-almost surely, it is sufficient to prove the second assertion. Using that $\Nhits(\infty) = \Zstop$ $\P^x$-almost surely,
\begin{align*}
1-\TW_s(x) &=\P^x[1-s^{\Nhits(\infty)}] = \P^x[(1-s^{\Nhits(\infty)}) \indic{ \Zstop>0}]\\
&= \P^x \big[ (1-s^{\Nhits(\infty)}) \frac{\Zstop(0) }{\Zstop}\frac{\Zstop}{ \Zstop(0)}   \indic{ \Zstop>0} \big] \\
&= \Q^x \big[\frac{\Zstop(0)}{ \Zstop}(1-s^{\Nhits(\infty)}) \indic{ \Zstop>0} \big] \\
&=\Q^x\big[ \frac{1-s^{\Nhits(\infty)}}{\Nhits(\infty)} \big] e^{-r x}.
\end{align*}
\end{proof}

Since we already know from Proposition \ref{prop3} (its proof is analytical and is independent of the present discussion) that $(1-\TW_s(x))e^{rx}$ tends to a constant $B>0$, it is now clear that the $\Q^x$ expectations in Lemma~\ref{18} also 
converge to $B$ as $x\to\infty$. However, we are now going to define $\Q^{\infty}$ as the law of the process under which we can couple all the $\Q^x$ together and interpret the limit constant $B$ as the expectation of a limit variable under $\Q^\infty$. Loosely speaking, we want $\Q^{\infty}$ to be the law of the process where the spine particle starts at $x=+\infty$  before drifting to 0. In fact it is easier to reverse time and have the spine start at 0 and drift to $+\infty.$

We start with the
following Lemma:
\begin{lem}\label{coupling}
Let $(Y(t), t\ge0)$ be a Brownian motion with drift $+\sqrt{\mu^2-2\beta}$,
started from 0 conditioned to never hit 0. Otherwise said, $Y$ is solution
of the following stochastic differential equation 
$$
\begin{cases}
\diffd Y(t) =\diffd B_t +\sqrt{\mu^2-2\beta}\coth \big( \sqrt{\mu^2-2\beta}
\, Y(t)\big)\diffd t \qquad \text{ if } \mu >\sqrt{2\beta} \\
\diffd Y(t) =\diffd B_t +{1\over Y(t) } \diffd t \qquad \text{ if } \mu =\sqrt{2\beta} .
\end{cases}
$$
Let $(t_i)$ be a Poisson point process on $\R^+$ with intensity $2\beta$.
For each $i\ge 0$ start a branching Brownian motion with law $\P^{Y(t_i)}$
and call $\tilde\Nhits_i$ the total number of absorbed particles  at 0 for
this process.

Fix $x>0$, then the distribution of the variable $\Nhits(\infty)$  under
$\Q^x$ is the same as that of 
$$
\Nhits^x(\infty):=1+\sum_{i : t_i \le \tau_x} \tilde\Nhits_i 
$$ 
under $\Q^\infty$ where $\tau_x := \sup_{t \ge 0} \{ Y(t)=x\}$.
\end{lem}
This result should be clear once it is realized that the process
$Y$ is the reversed path of the spine $\xi$.

\begin{proof} We only treat the case $\mu>\sqrt{2\beta}$ since the zero-drift case is similar.
The only thing we need to prove here is that if $(\xi(t), t\le \tau_\xi)$ is
a Brownian motion with drift $-\sqrt{\mu^2-2\beta}$ started form $x$ and
stopped at time $\tau_\xi := \inf\{ t: \xi(t)=0\}$, then 
$$
\{(\xi(t), t\le \tau_\xi), \tau_\xi\}  \stackrel{\mathcal{L}}{=} \{(Y(\tau_x-t), t\le \tau_x), \tau_x\}.
$$
This follows for instance from \cite{larbi}.

\end{proof}

The upshot of Lemma \ref{coupling} is that we can now construct the
variables $\Nhits(\infty)$ under $\Q^x$ for all values of $x$
simultaneously. We write $\Q^\infty$ for the joint law of the variables
$((Y(t), t\ge 0), (t_i)_{i \in \N}, \tilde\Nhits_i)$ described above.
Then under $\Q^\infty$, clearly
$(\Nhits^x(\infty), x\ge 0)$ is an increasing process in $x$. We call $\Nhits^\infty(\infty)$ its limit which is also the total number of particles absorbed at 0 under $\Q^\infty.$

\begin{lem}\label{conv of Nx}
We have that 
$
\Nhits^\infty(\infty)  <\infty
$
$\Q^\infty$-almost surely.
\end{lem}
\begin{proof}
We start with the $\mu>\sqrt{2\beta} $ case.
First observe that for $\epsilon>0$ fixed, there exists almost surely a random $i_0 \in \N$ such that 
$$
\forall i\ge i_0,
\quad Y(t_i) \ge \Big({\sqrt{\mu^2-2\beta} \over 2\beta}-\epsilon \Big)
i  .
$$
This simply comes form the fact that $Y(t)/t \to \sqrt{\mu^2-2\beta}$ and
$t_i/i \to (2\beta)^{-1}$ almost surely. Now, $1-\TW(x) \le e^{-(\mu
+\sqrt{\mu^2-2\beta})x}$ because under $\Q^x$ we have that $\Zstop\ge 1$
almost surely and therefore $\Q^x(1/\Zstop)\le 1.$ Hence, for any $i\ge
i_0$ we have that 
$$
\Q^\infty(\tilde \Nhits_i >0) \le 1-\TW  \Big(\big({\sqrt{\mu^2-2\beta} \over 2\beta}-\epsilon \big) i \Big) \le e^{-c  i }
$$
for some positive constant $c$. Thus a straightforward application of Borel-Cantelli Lemma shows that almost surely, there exists $j_0\in \N$ such that $\forall i\ge j_0, \tilde \Nhits_i =0,$ which yields the desired result. The zero-drift case is similar. One just needs to start the argument by observing that for $\epsilon>0$ fixed, there exists almost surely a random $i_0 \in \N$ such that 
$$
\forall i\ge i_0,
\quad Y(t_i) \ge c i^{1/2-\epsilon}
$$
where $c$ is a constant. The proof then follows as before.
\end{proof}

The following Lemma completes the proof of Theorem \ref{Thm: martingale representation}.
\begin{lem}\label{Q^x to Q infty}
For $ s\in [0,s_0]$ small enough, we have that 
\begin{align}
\label{Q2} \Q^x \Big(\frac{s^{\Nhits(\infty)}}{\Nhits(\infty)}\Big) &\to  \Q^\infty
\Big(\frac{s^{\Nhits^\infty(\infty)}}{\Nhits^\infty(\infty)}\Big).
\end{align}
\end{lem}

\begin{proof}
The monotone convergence Theorem  applies when $s\le 1$
so we suppose $1<s\le s_0$. Observe that the map $t \mapsto s^t/t$ is
decreasing on $[1,1/\log s]$ and increasing on $[1/\log s, \infty).$ Thus 
we write
\begin{align*}
\Q^x \Big(\frac{s^{\Nhits(\infty)}}{\Nhits(\infty)}\Big) &= \Q^{\infty}
\Big(\frac{s^{\Nhits^x(\infty)}}{\Nhits^x(\infty)}\Big) \\
&= \Q^{\infty} \Big(\frac{s^{\Nhits^x(\infty)}}{\Nhits^x(\infty)}
;   \Nhits^x(\infty) \le 1/\log s   \Big)  +\Q^{\infty}
\Big(\frac{s^{\Nhits^x(\infty)}}{\Nhits^x(\infty)}     ;   \Nhits^x(\infty) \ge 1/\log s   \Big) \\
&\to \Q^{\infty}
\Big(\frac{s^{\Nhits^\infty(\infty)}}{\Nhits^\infty(\infty)}
;   \Nhits^\infty(\infty) \le 1/\log s   \Big)  +\Q^{\infty}
\Big(\frac{s^{\Nhits^\infty(\infty)}}{\Nhits^\infty(\infty)}
;   \Nhits^\infty(\infty) \ge 1/\log s   \Big) 
\\&=\Q^{\infty} \Big(\frac{s^{\Nhits^\infty(\infty)}}{\Nhits^\infty(\infty)}   \Big)
\end{align*}
where the first convergence comes from the dominated convergence Theorem
and the second from the monotone convergence Theorem.

\end{proof}


\subsubsection{Proof of Proposition~\ref{prop3}}
\label{proofprop3}

We consider here all the solutions $x\mapsto v(x)$ to \eqref{E:TW} that
remains in $[0,1)$. By Cauchy's theorem, a solution to \eqref{E:TW} is
entirely determined once the derivative at the origin is given.

Let $r$ and $R<r$ be the two roots of the polynomial $\frac12X^2-\mu X+\beta$:
$$r=\mu+\sqrt{\mu^2-2\beta},\qquad R=\mu-\sqrt{\mu^2-2\beta}.$$
(See also \eqref{eq:r}.) From the general theory of differential equations,
one has
\begin{lem}
\label{onlyone}
Let $v$ be a solution to
\begin{equation}\label{eqv2}
0=\frac12v''+\mu v'+\beta(v^2-v)
\end{equation}
such that $v(x)$ converges to~1 as $x\to\infty$. Then, for some non-zero
constant $A$ or $B$,
\begin{itemize}
\item if $\mu>\sqrt{2\beta}$ one has either $1-v(x)\sim A e^{-Rx}$ or
$1-v(x)\sim B e^{-rx}$ as $x\to\infty$.
\item If $\mu=\sqrt{2\beta}=r=R$ one has either $1-v(x)\sim A x e^{-\mu x}$ or
$1-v(x)\sim B e^{-\mu x}$ as $x\to\infty$.
\end{itemize}
Furthermore, up to invariance by translation, there are exactly two
solutions which converges to~1 in the fast way (as $B e^{-rx}$); one of
them approaching 1 by above and the other from below.
\end{lem}
This lemma simply tells that the solutions to the non-linear
equation~\eqref{eqv2} behave around $v=1$ as the solutions to the 
equation linearised around~1. We already used implicitly this result in
Section~\ref{KPPasymp}.
\begin{proof}
This  follows from a result of Hartman \cite{hartman} who shows that if
\begin{equation}
\dot{X} = \Gamma X +F(X)
\end{equation}
is a non-linear differential system of dimension 2 with $\Gamma$ a hyperbolic (eigenvalues have a non-zero real part) matrix and $F$ is $C^1$ with $F(x) = o(|x|)$ as $x\to 0$ (so 0 is a critical point) , then there exists a $C^1$ diffeomorphism $\phi$ with derivative the identity at the origin such that $U(t) = \phi(X(t))$ solves 
\begin{equation}
\dot{U} =\Gamma U.
\end{equation}
Otherwise said the solutions of the linearized system and the solutions of the non-linear system are (locally around 0) in one-to-one correspondence through $\phi$.

We apply this result to the following system where $v=1-u$ is a solution of \eqref{eqv2}
\begin{equation}
X(t) =  \begin{pmatrix} x(t) \\ y(t) \end{pmatrix} = \begin{pmatrix} u(t) \\ u'(t) \end{pmatrix} ;  \quad \dot{X}(t)= \begin{pmatrix} u'(t) \\ u''(t) \end{pmatrix} =  \begin{pmatrix} y(t) \\ -2\mu y(t)-2\beta \big[x(t)-x^2(t)\big] \end{pmatrix} ,
\end{equation}
which has a critical point at $(x,y)=(0,0)$. In this case 
\begin{equation}
\Gamma = \begin{pmatrix} 0 &1\\-2\beta & -2\mu \end{pmatrix}
\end{equation}
with eigenvalues $- r = -\mu -\sqrt{\mu^2-2\beta}$ and $-R = -\mu +\sqrt{\mu^2-2\beta}$ (for simplicity we only consider the case $r\neq R$ here) and corresponding eigenvectors $\begin{pmatrix} 1 \\-r \end{pmatrix}$ and $\begin{pmatrix} 1 \\-R \end{pmatrix}$. The solutions of  $\dot U =\Gamma U$ are of the form 
\begin{equation}
U(t)=\begin{pmatrix} u_1(t) \\ u_2(t) \end{pmatrix} = Be^{-rt} \begin{pmatrix} 1 \\-r \end{pmatrix} + Ae^{-Rt} \begin{pmatrix}1 \\-R \end{pmatrix}.
\end{equation}
Thus the only solutions such that $|U(t)| \sim ce^{-rt}$ for some constant $c$ are those such that $A=0.$ If $B>0$ then  $u_1$ approaches by above, if $B<0$ $u_1$ approaches by below. Hartman's Theorem tells us that there exists
\begin{equation}
\phi(X) =X + f(X), \quad f(X)=o(|X|) \text{ when } x\to 0
\end{equation}
such that the solutions $X(t)$ of the non-linear system are locally 
\begin{equation}
X(t)=\phi^{-1}(U(t)).
\end{equation}
Thus, (after a shift in the argument, replacing $x$ by $x+\ln |B| /r$) there is exactly one solution $X$ to the non linear system such that $|X(t)|e^{rt} $ has a non degenerate limit and such that $x(t)$, the first coordinate of $X(t)$, is eventually positive (resp. eventually negative). 
\end{proof}

Let $s<1$ be such that there is a solution $v$ of \eqref{eqv2} with $v(0)=s$, $1-v(x) \sim c e^{-rx}$, and $v(x)<1$ for all $x\ge 0$ (we now know that such an $s$ exists). Then $\TW_s(x)$ being the maximal solution of  \eqref{E:TW} that starts from $s$ and stays below 1, we must have 
$$
\TW_s(x) \ge v(x), \forall x\ge 0.
$$
Since we also know that the only two possibilities for the asymptotic behavior of $\TW_s$ are that either 
$\TW_s(x) e^{rx} \to c$ or $\TW_s(x) e^{Rx} \to c$ we conclude that it is the former that holds. The same argument apply for $\TW_s(x)$ for any $s\le s_0$ and in the critical case.
This concludes the proof of
Proposition~\ref{prop3} for $\TW(x)$.


\subsubsection{Proof of Proposition~\ref{thm:solutionExpansion}}
\label{sub:Expansion-formulaProof}
We consider the series $\Phi(z)=\sum_{n\ge1}a_n z^n$ defined in~\eqref{Phi}
with the coefficients~$a_n$ defined in~\eqref{eq:expansionCoefficients}.
The function $z\mapsto\Phi(z)$ is a well defined object because, by
induction on \eqref{eq:expansionCoefficients} one has easily $0<a_n\le1$
and, therefore, $\mathcal R\ge1$.
It is then very easy to check by direct substitution that for any $B\in\R$,
the function
\begin{equation}\label{vB}
x\mapsto v(x) = 1-\Phi(B e^{-rx})\quad\text{for $x$ such that
$|B|e^{-rx}<\mathcal R$},
\end{equation}
is solution to the partial differential equation $\frac12v''+\mu
v'+\beta(v^2-v)=0$ which appears in \eqref{E:TW} (when discussing $\TW$)
and in \eqref{Eq:TW_s} (when discussing $\TW_s$).
Recall that $r=\mu+\sqrt{\mu^2-2\beta}$ is the larger root of
$\frac12X^2-\mu X+\beta$.

As the coefficients $a_n$ are positive, $\Phi(z)$ is
non-negative and increasing for $z\ge0$. As $a_1=1$ and $a_2>0$, it is easy
to find a $0<z_0<1\le\mathcal R$ such that $\Phi(z_0)>a_1 z_0+a_2z_0^2>1$.
This implies that there must exists a $B_0\in(0,\mathcal R)$ (smaller than
$z_0$) such that
$\Phi(B_0)=1$. With $B=B_0$, the function~$v(x)$ in \eqref{vB} is smaller
than~1 and converges to~1 for large~$x$ as $e^{-rx}$. Using
Proposition~\ref{prop3}, this implies that $v(x)=\TW(x)=1-\Phi(B_0
e^{-rx})$.

Recall by Theorem~\ref{thm w_s} that $\TW_s$ for $s<1$ is simply equal to
$\TW$ correctly shifted to have $\TW_s(0)=s$. This implies that, for $s<1$,
$\TW_s(x)=1-\Phi(B_s e^{-rx})$ where $B_s\in(0,B_0]$ is such that
$\Phi(B_s)=1-s$.

The case $s=1$ is trivial, we now turn to $s>1$. As for $s=0$, we have the following
points:
\begin{itemize}
\item for $s>1$, $\TW_s$ is the smallest solution to
\eqref{Eq:TW_s} that remains above~1 (Theorem~\ref{thm2}).
\item By Lemma~\ref{onlyone}, there is exactly one solution to
\eqref{Eq:TW_s} which remains above~1 and decays to 1 as $e^{-rx}$.
Because of the previous point, this solution must be $\TW_s$.
\end{itemize}
Now consider $\Phi(z)$ for negative arguments. Because $\Phi(0)=0$ and
$\Phi'(0)=1$, there must exists $B\in(-\mathcal R,0)$ such that $\Phi$ is
negative on $[B,0)$. Then, the function $x\mapsto 1-\Phi(B e^{-rx})$ is
solution to~\eqref{Eq:TW_s} for $s=1-\Phi(B)>1$, remains above~1 for
$x\ge0$ and converges to~1 as~$e^{-rx}$.
Therefore, it must be $\TW_s$ for that particular~$s$.

But all the functions~$\TW_s$ for $1<s\le s_0$ are related through
Theorem~\ref{thm w_s}: they are all shifted versions of $\TW_{s_0}$.
Therefore, for any $s\in(1,s_0]$, one has $\TW_s(x)=1-\Phi(B_s e^{-rx})$
for a well chosen negative~$B_s$ (which represents the shift), at least for
values of $x$ sufficiently large to have $|B_s|e^{-r x}<\mathcal R$.


\subsection{Proof of Theorem \ref{thm:KPP}}
\label{S:proof mu<muc}
We assume to be in regime A or B ($\mu <\sqrt{2\beta}$) and we want to show how
$\pr(t,x)=P^x(K(t)=0)$ converges to a KPP travelling wave.

The proof is essentially analytic and relies on Bramson's result \cite{bramson:1983fk} and the maximum principle.
The key step is to compare $\pr(t,x)$ to a new function
$\auxsol^{T}:[T,+\infty)\times\R\mapsto\R$ (where $T\ge 0$ is a parameter)
where $\auxsol^T(t,x)$  is defined as the probability, in the standard
branching Brownian motion (without absorption nor stopping) with
drift~$\mu$ starting
from  $x$, that no particles are present 
in the negative half-line between times $t-T$ and $t$.   In symbols
\begin{equation}
\auxsol^{T}(t,x):=\P^x\big(\forall r\in  [t-T, t], \forall u \in \Nall(r)
: X_u(r) > 0 \big),
\label{eq:representation}
\end{equation}
where we recall that $\Nall(s)$ is the population of particles at time $s$ in a branching Brownian motion without absorption or stopping.

The advantage of $\auxsol^T$ is that since it is defined on $\Nall$ it satisfies a KPP equation on the whole real line:  
\[
\begin{cases}
\partial_t
\auxsol^{T}=\frac{1}{2}\partial_{xx}\auxsol^{T}
+\mu\partial_x\auxsol^{T}
+\beta\big((\auxsol^{T})^{2}-\auxsol^{T}\big),
\qquad (t,x) \in [T,+\infty)\times\R  \\
\auxsol^{T}(T,x)=\pr(T,x),\quad\text{for }x\geq0,\\
\auxsol^{T}(T,x)=0,\quad\text{for }x<0.
\end{cases}
\]
Otherwise said the function $\tilde \auxsol^T(t,x) =\auxsol^T(T+t, x)$
solves the KPP equation on the whole line with initial condition $\tilde
\auxsol^T(0,x) = \pr(T,x)\indic{x>0}$. Since for $T>0$ fixed, $1- \pr(T,x)$
goes to 0 as $x \to
\infty$  with a super exponential  decay
(like the tail of a Gaussian), Bramson's convergence Theorem  \cite[Theorem
A]{bramson:1983fk},
ensures that there exists a constant $\tilde C_{T} \in\R$  such that we have 
\begin{equation}\label{tilde u^T converge}
\Vert \tilde \auxsol^{T}(t,\cdot+m_t-\mu t +\tilde
C_{T})-h_*(\cdot)\Vert_\infty\to 0 \text{ as } t \to \infty,
\end{equation}
where Bramson's displacement $m_t$ is given in \eqref{m_t}.
The value $\tilde C_T$  depends on $T$ because for different  $T$ we plug
different initial conditions in the KPP equation .

Since $m_t -m_{t-T} \to \sqrt{2\beta}\,T$ when $t\to \infty$, one obtains
taking $C_T =\tilde C_T - \sqrt{2\beta}\, T+\mu T $:
\begin{equation}\label{u^T converge}
\Vert \auxsol^{T}(t,\cdot+m_t-\mu t+C_{T})-h_*(\cdot)\Vert_\infty\to 0 \text{ as
} t \to \infty.
\end{equation}

Therefore we only need to show that for $t$ large enough, $\pr(t,x)$ is close to
$\auxsol^{T}(t,x)$:
\begin{lem}\label{u^T and w are close}
 \begin{equation}\label{E:U^T converge to w} 
 \Vert\auxsol^{T}(\cdot,\cdot)-\pr(\cdot,\cdot)\Vert_\infty\to 0
\text{ as } T \to \infty.
 \end{equation} 
In addition, there exists $C\in \R$ such that  $C_T \to C$ as $T\to \infty.$
\end{lem}

Indeed, assuming that Lemma \ref{u^T and w are close} holds, we can conclude the
\begin{proof}[Proof of Theorem \ref{thm:KPP}]
Fix $\epsilon >0$. Using (\ref{u^T converge}) and
\eqref{E:U^T converge to w},  choose $T$ large enough that
$\Vert\auxsol^{T}(\cdot,\cdot)-\pr(\cdot,\cdot)\Vert_\infty<\epsilon$ and then
choose $t$ large enough so that $\Vert
\auxsol^{T}(t,\cdot+m_t-\mu t+C_{T})-h_*(\cdot)\Vert_\infty<\epsilon$. Then,  we
have that
\begin{align*}
\Vert\pr(t,\cdot+m_t-\mu t)  -h_*(\cdot-C)\Vert_\infty
  & \leq\Vert\pr(t,\cdot+m_t-\mu t)-\auxsol^{T}(t,\cdot+m_t-\mu t)\Vert_\infty
\\&\quad +\Vert\auxsol^{T}(t,\cdot+m_t-\mu t)-h_*(\cdot-C_{T})\Vert_\infty
\\&\quad +\Vert h_*(\cdot-C_{T})-h_*(\cdot-C)\Vert_\infty,\\
 & \leq2\epsilon+c|C_{T}-C|,
\end{align*}
where $c = \max_{x \in \R} h_*'(x)$. As $C_T\to C$, for $T$ large enough
independently of $x$ this can be made smaller than $3\epsilon.$ Thus   
$
\Vert \pr(t,\cdot+m_t-\mu t)  -h_*(\cdot-C)\Vert_\infty\to 0$  as 
$t \to \infty$,
which is the Theorem.
\end{proof}

It now remains to prove Lemma \ref{u^T and w are close}.
We start with
\begin{lem}
\label{lem:maximalAt0}For any $\epsilon>0$ there exists $T_\epsilon$ such
that for all $t\ge T\ge T_\epsilon$ one has
$\auxsol^{T}(t,0)\leq\epsilon/(1+\epsilon).$\end{lem}
(The $1+\epsilon$ in the denominator makes the following easier.)
\begin{proof}
We use the representation (\ref{eq:representation}). Let $t\geq T$;
obviously
\[
\auxsol^{T}(t,0)\leq\mathbb{P}^{0}(\min_{ u \in  \Nall(t)} X_{u}(t) > 0)=
 h(t,-\mu t),
\]
where $h$ is the solution of \eqref{reallyKPP}. $h(t,-\mu t)$ is by
definition the
probability that the leftmost particle at time $t$ of a  
driftless branching Brownian motion is to the right of $-\mu t$; it is also
the probability that the leftmost particle at time $t$ of a branching
Brownian motion with drift $\mu$ is to  the right of zero. For $\mu<
\sqrt{2\beta}$ (regimes A and B), this probability is known to tend to zero
when $t\to \infty$. 
\end{proof}
The next step is the following Lemma:
\begin{lem}
\label{lem:supDistance}For any $\epsilon>0$ and any $T>T_\epsilon$ one has
\[
(1+\epsilon)\auxsol^{T}(t,x)-\epsilon\leq\pr(t,x)\leq\auxsol^{T}(t,x),\quad(t,x)\in[T,\infty)\times\R_{+}.
\]
\end{lem}
(The $T_\epsilon$ in Lemma~\ref{lem:supDistance} is the same as in 
Lemma~\ref{lem:maximalAt0}.)
\begin{proof}
$\pr \leq \auxsol^T$ follows immediately from their definitions as probabilities. Let us introduce
\[
\tilde{\pr}(t,x):=\frac{\pr(t,x)+\epsilon}{1+\epsilon}.
\]
We have that
\[
(1+\epsilon)\partial_t\tilde{\pr}
=(1+\epsilon)\frac{1}{2}\partial_{xx}\tilde{\pr}
+(1+\epsilon)\mu\partial_x\tilde{\pr}+\beta\big[\big(
(1+\epsilon)\tilde{\pr}-\epsilon\big)^{2}-
\big((1+\epsilon)\tilde{\pr}-\epsilon\big)\big]
\]
Performing simple calculations we arrive at 
\begin{align*}
\partial_t\tilde{\pr}&=\frac{1}{2}\partial_{xx}\tilde{\pr}
+\mu\partial_x\tilde{\pr}+\beta(\tilde{\pr}-1)(\tilde{\pr}-\epsilon+\epsilon\tilde{\pr})\\
&\ge \frac{1}{2}\partial_{xx}\tilde{\pr}+
\mu\partial_x\tilde{\pr}+\beta (\tilde{\pr}-1)\tilde{\pr}
\end{align*}
since $\tilde u\le1$ and
$\epsilon>0$.

Now, for any $T>T_\epsilon$, we have with Lemma \ref{lem:maximalAt0}
\[
\auxsol^{T}(t,0)\leq \frac\epsilon{1+\epsilon} =
\tilde{\pr}(t,0)
,\qquad t\geq T.
\]
Moreover one checks directly that 
\[
\auxsol^{T}(T,x) =\pr(T,x) \le \tilde{\pr}(T,x) ,\quad x\geq0.
\]
By the parabolic maximum principle (and the unicity of solutions)
\cite{BerestyckiHamel}
we get that for any $T>T_\epsilon$
\[
\auxsol^{T}(t,x) \le \tilde{\pr}(t,x),\quad   \forall(t,x)\in(T,\infty)\times\R_{+}.
\]
This proves the first inequality and thus concludes the proof of the lemma.
\end{proof}
Lemma~\ref{lem:supDistance} implies that
$|\pr(t,x) -\auxsol^T(t,x)| \le \epsilon (1-\auxsol^T(t,x)) \le \epsilon$
for each $x\in \R_+$ and each $t$ and $T$ with $t\ge T \ge T(\epsilon)$,
which is the first assertion of Lemma~\ref{u^T and w are close}.

The last step is then to prove that $C_T$ has a limit~$C$ for large~$T$.

As $\pr(t,\cdot)$ is strictly increasing and continuous,
$\pr(t,0)=0$ and $\lim_{x\to+\infty}\pr(t,x)=1$, we may define $m_{\frac12}:(0,+\infty)\mapsto\R_{+}$
by 
\[
\pr\big(t,m_{\frac12}(t)\big)=1/2.
\]
Fix $\epsilon>0$. We have that
\begin{align*}
\big|1/2-h_*\big(m_{\frac12}(t)-m_t+\mu t-C_{T}\big)\big|
&\le
\big|\pr\big(t,m_{\frac12}(t)\big)-\auxsol^{T}\big(t,m_{\frac12}(t)\big)\big|
\\&\quad+\big|\auxsol^{T}\big(t,m_{\frac12}(t)\big)
		-h_*\big(m_{\frac12}(t)-m_t+\mu t-C_{T}\big)\big|,
\\&\leq2\epsilon,
\end{align*}
as long as $T$ and $t$ are large enough by  \eqref{u^T converge} and
\eqref{E:U^T converge to w}. From this we deduce 
\begin{equation}
m_{\frac12}(t)-m_t+\mu t-C_{T}\in  \big[ h_*^{-1}(1/2-2\epsilon),
h_*^{-1}(1/2+2\epsilon) \big].\label{eq:keyEstimate1}
\end{equation}
Consequently
\[
\limsup_{t\to+\infty}\big[m_{\frac12}(t)-m_t+\mu t\big]
-\liminf_{t\to+\infty}\big[m_{\frac12}(t)-m_t+\mu t\big]
\leq
h_*^{-1}(1/2+2\epsilon)-h_*^{-1}(1/2-2\epsilon).
\]
Since $\epsilon$ can be chosen arbitrarily small we have
that $\lim_{t\to+\infty}\big[m_{\frac12}(t)-m_t+\mu t\big]=C$, for
some constant $C\in\R$.
This and  (\ref{eq:keyEstimate1}) immediately yields that 
\[
\lim_{T\to+\infty}C_{T}=C,
\]
where we used that $h_*^{-1}(1/2)=0$.
This concludes the proof of Lemma~\ref{u^T and w are close}.


\section{Radius of convergence and asymptotic behavior of $S_0$}
\label{furthers0}

In Section~\ref{series}, we related the $\TW_s(x)$ to a function
$x\mapsto\Phi(z)$ defined as a series of which the coefficients $a_n$
follows the recursive equation~\eqref{eq:expansionCoefficients}. We write here
the same property in a slightly different but equivalent way. Let
$p\in(0,1]$  be defined by
$$
p:=\frac{2\beta}{r^2},
$$
and introduce $\Psi^{(p)}(z)=p\Phi(z/p)$ and $b_n^{(p)}=a_n/p^{n-1}$. These
quantities satisfy the relations
\begin{equation}
\Psi^{(p)}(z)=\sum_{n\ge1}b_n^{(p)}z^p,\qquad
b_1^{(p)}=1,\qquad  b_n^{(p)}=\frac{1}{(n-1)(n-p)}\sum_{j=1}^{n-1}
b_j^{(p)} b_{n-j}^{(p)},\quad n\ge2.\label{eq:rescaledFunction}
\end{equation}
Let $\mathcal R^{(p)}$ be the radius of convergence of $\Psi^{(p)}$.
We know that there exists a $B_{s_0}$ relating $\Psi^{(p)}$ and $\TW_{s_0}$
through
$$\TW_{s_0}(x)=1-\frac1p\Psi^{(p)}(p B_{s_0} e^{-rx}).$$
The following observation will be useful. Since $\TW_{s_0}'(0)=0$ and $\TW_{s_0}''(0)<0$, the function $\TW_{s_0}$ (defined on a domain containing zero) has a local maximum in zero. This implies that for $p>0$ the function $\Psi^{(p)}$ has a local minimum in $m^{(p)}:=p B_{s_0}<0$. In fact $m^{(p)}$ is the first local minimum (and indeed the first point where the first derivative cancels) one encounters left of zero for $\Psi^{(p)}.$



The steps of the proof are the following: 
\begin{enumerate}
\item We show that ${\mathcal R}^{(p)}\ge 4$  for small enough $p$ (including $p=0$).
\item We prove  that there exists  $m^{(0)}  \in (-3,0]$ which is the first minimum one encounters left of zero for $\Psi^{(0)}$ and that 
\begin{equation}
	[\Psi^{(0)}]'(x)<0, x\in [a,m^{(0)}),\quad \text{and}\quad [\Psi^{(0)}]'(x)>0, x\in (m^{(0)},0],\label{eq:derivativeControl}
\end{equation}
for some $a\in[-3,m^{(0)}]$.
\item We show that $[\Psi^{(p)}]'$ converges to $[\Psi^{(p)}]'$ uniformly on $(-3,0]$. This implies that 
\begin{equation}
	\lim_{p\searrow 0}m^{(p)}= m^{(0)} \in(-3,0).\label{eq:mpconvergence}
\end{equation}
\item 
Since $|pB_{s_0}| \to |m^{(0)}|<4$ we conclude that  that $B_{s_0}$ is within the radius of convergence of $\Phi$ for $p$ small enough. 
The identity
$$
s_0 = w_{s_0}(0) =1-\Phi(B_{s_0}) =1-\Psi^{(p)}(m^{(p)})/p.
$$
shows that 
\[
	\lim_{p\searrow 0} p s_0(p) = \Psi^{(0)}(m^{(0)}),
\]
where we made the dependance of $s_0$ on $p=2\beta/r^2$ explicit. 
\end{enumerate}
We now prove these points.
\begin{enumerate}
\item The key remark is that if for a real $\alpha>0$ and an integer
$n_0$, one has $b_n^{(p)}\le (n_0-p) \alpha^{-n}$ for all $n\in\{1,\ldots,
n_0-1\}$ then, as can be shown by a very simple recursion, the property
$b_n^{(p)}\le (n_0-p) \alpha^{-n}$ holds for all~$n\ge1$.

Computing the first values of $b_n^{(0)}$, one checks easily that the
maximum of $4^n b_n^{(0)}$ for $n\in\{1,\ldots,14\}$ is around 14.14. 
For $p$ small enough, by
continuity of $p\mapsto b_n^{(p)}$, the maximum of $b_n^{(p)}$ for
$n\in\{1,\ldots,14\}$ will be no more than $15-p$ and hence one has
\begin{equation}
b_n^{(p)}\le 15\times4^{-n},\qquad\text{(for $p$ small enough)}
\label{boundb}
\end{equation}
As a consequence, $\mathcal R^{(p)}\ge4$ for $p$ small enough (including $p=0$).

\item The bound \eqref{boundb} applies for $p=0$. 
Thus, for any $z\in[-3,3]$ using only the 
the fifty first terms of the expansion leads to an error of at most
$\sum_{n\ge51} 15\times(3/4)^n<3\,10^{-5}$. In that way we computed
$\Psi^{(0)}(-3)\approx -0.8528$ and $\Psi^{(0)}(-2.5)\approx-0.8575$.
Therefore $\Psi^{(0)}(-2.5)$ is smaller than both $\Psi^{(0)}(-3)$ and
$\Psi^{(0)}(0)=0$, and the function $\Psi^{(0)}$ must have a minimum in
$(-3,0)$. In other words we proved $m^{(0)} \in (-3,0)$. It is easy to check that $\sum_{n=1}^{50}  n(n-1) a_nx^{n-2}\geq 0.7$ for $x\in [-3,0]$. Estimating an error by $\sum_{n\geq 51} 15 n(n-1) (3/4)^n <0.074 $ we conclude that $[\Psi^{(0)}]''(x)>0$ for $x\in [-3,0]$. In this way we get \eqref{eq:derivativeControl}.

\item By \eqref{boundb} there exist $p_0>0$ and $C>0$ such that  the functions $[\Psi^{(p)}]'$
are analytic in $[-3,0]$ and $\sup_{p\in[0,p_0],x\in[-3,0]} |\Psi^{(p)}(x)|<C$. By
continuity of $p\mapsto b_n^{(p)}$, for any
$x\in[-3,0]$ we have $[\Psi^{(p)}]'(x)\to[\Psi^{(0)}]'(x)$. The
Vitali-Proter theorem strengthen this to uniform convergence. This together with \eqref{eq:derivativeControl} implies easily \eqref{eq:mpconvergence}.
\end{enumerate}

%
%
%
%
%
\subsection*{Acknowledgments} PM's research was supported by NCN grant DEC-2012/07/B/ST1/03417".

\bibliographystyle{abbrv}
\bibliography{bibliography}

\end{document}

%% file: regimes.pdf_t
\begin{picture}(0,0)%
\includegraphics{regimes.pdf}%
\end{picture}%
\setlength{\unitlength}{2605sp}%
\begingroup\makeatletter\ifx\SetFigFont\undefined%
\gdef\SetFigFont#1#2#3#4#5{%
  \reset@font\fontsize{#1}{#2pt}%
  \fontfamily{#3}\fontseries{#4}\fontshape{#5}%
  \selectfont}%
\fi\endgroup%
\begin{picture}(11277,2490)(211,-4489)
\put(3526,-3886){\makebox(0,0)[lb]{\smash{{\SetFigFont{8}{9.6}{\rmdefault}{\mddefault}{\updefault}{\color[rgb]{0,0,0}$t$}%
}}}}
\put(8476,-4411){\makebox(0,0)[lb]{\smash{{\SetFigFont{8}{9.6}{\rmdefault}{\mddefault}{\updefault}{\color[rgb]{0,0,0}Regime C : $\mu \ge \sqrt{2\beta}$}%
}}}}
\put(3826,-2986){\makebox(0,0)[lb]{\smash{{\SetFigFont{8}{9.6}{\rmdefault}{\mddefault}{\updefault}{\color[rgb]{0,0,0}$x$}%
}}}}
\put(7276,-3886){\makebox(0,0)[lb]{\smash{{\SetFigFont{8}{9.6}{\rmdefault}{\mddefault}{\updefault}{\color[rgb]{0,0,0}$t$}%
}}}}
\put(4276,-4411){\makebox(0,0)[lb]{\smash{{\SetFigFont{8}{9.6}{\rmdefault}{\mddefault}{\updefault}{\color[rgb]{0,0,0}Regime B : $-\sqrt{2\beta} < \mu <\sqrt{2\beta}$}%
}}}}
\put(11251,-3886){\makebox(0,0)[lb]{\smash{{\SetFigFont{8}{9.6}{\rmdefault}{\mddefault}{\updefault}{\color[rgb]{0,0,0}$t$}%
}}}}
\put(7726,-2986){\makebox(0,0)[lb]{\smash{{\SetFigFont{8}{9.6}{\rmdefault}{\mddefault}{\updefault}{\color[rgb]{0,0,0}$x$}%
}}}}
\put(676,-4411){\makebox(0,0)[lb]{\smash{{\SetFigFont{8}{9.6}{\rmdefault}{\mddefault}{\updefault}{\color[rgb]{0,0,0}Regime A : $\mu \le -\sqrt{2\beta}$}%
}}}}
\put(226,-2836){\makebox(0,0)[lb]{\smash{{\SetFigFont{8}{9.6}{\rmdefault}{\mddefault}{\updefault}{\color[rgb]{0,0,0}$x$}%
}}}}
\end{picture}%